\documentclass[12pt]{amsart}
\usepackage[margin={0.7in}, bmargin={0.7in}, tmargin={0.7in}  ]{geometry}
\usepackage{amsfonts}
\usepackage{amsmath}
\usepackage{epsfig}
\usepackage{color}
\usepackage{graphicx}
\usepackage{psfrag}
\usepackage{eufrak}
\usepackage{subfigure}
\usepackage{geometry}
\usepackage{mathtools}
\usepackage{mhequ}
\usepackage{comment}
\usepackage{mathrsfs}
\usepackage{soul}
\usepackage{hyperref}
\usepackage{stix}
\mathtoolsset{showonlyrefs}
\usepackage[normalem]{ulem}
\definecolor{verdeosc}{rgb}{0,0.6,0}
\numberwithin{equation}{section}

\newcommand{\be}{\begin{equation}}
\newcommand{\ee}{\end{equation}}

\DeclareMathOperator{\Tail}{Tail}

\def\r{\mathbb R}

\newcommand{\red}[1]{{\color{red}#1}}

\newcommand{\pa}{\partial}

\newcommand{\ve}{\varepsilon}
\newcommand{\R}{\mathbb{R}}
\newcommand{\vp}{\varphi}
\newcommand{\D}{\Delta}
\newcommand{\LL}{\mathcal{L}}

\numberwithin{equation}{section}

\newtheorem{lemma}{Lemma}[section]

\newtheorem{theorem}[lemma]{Theorem}
\newtheorem{proposition}[lemma]{Proposition}

\theoremstyle{remark}
\newtheorem{remark}[lemma]{Remark}

\begin{document}

\title{The limiting case of the fractional Caffarelli-Kohn-Nirenberg inequality in dimension one}

\author[M.d.M. Gonz\'alez]{Mar\'ia del Mar Gonz\'alez}

\address{Mar\'ia del Mar Gonz\'alez
\hfill\break\indent
Universidad Aut\'onoma de Madrid
\hfill\break\indent
Departamento de Matem\'aticas, and ICMAT.  28049 Madrid, Spain}
\email{mariamar.gonzalezn@uam.es}

\author[A. Hyder]{Ali Hyder}

\address{Ali Hyder
\hfill\break\indent
TIFR Centre for Applicable Mathematics, Sharadanagar, Yelahanka New Town, Bangalore 560065, India}
\email{hyder@tifrbng.res.in}

\author[M. S\'aez]{Mariel S\'aez}

\address{Mariel S\'aez
\hfill\break\indent Facultad de Matem\'atica, P. Universidad Cat\'olica de Chile,
\hfill\break\indent Av. Vicu\~na Mackenna 4860, 690444 Santiago, Chile. }
\email{\tt mariel@mat.uc.cl}

\maketitle

\begin{abstract}

 In this paper we study the fractional Caffarelli-Kohn-Nirenberg inequality (CKN)  in one dimension when the parameter $\gamma$ converges (from the left) to its critical value $1/2$,  obtaining   Onofri's inequality in the unit disk as the limit.  A  difficulty that we encounter is the lack of an explicit expression for the extremal function at which the CKN inequality is attained, which we address by studying  solutions of the weighted Liouville equation for the half-Laplacian in dimension one.

\end{abstract}

\section{Introduction and statement of the results}

Let $n\geq 2$ and $\gamma\in(0,1)$, or $n=1$ and $\gamma\in(0,\frac{1}{2})$. We fix  $\alpha,\beta\in\r$ satisfying
\begin{equation}\label{parameter}
\alpha\leq \beta\leq \alpha+\gamma, \quad -2\gamma<\alpha<\frac{n-2\gamma}{2},
\end{equation}
and set
\begin{equation}
p=\frac{2n}{n-2\gamma+2(\beta-\alpha)}. \label{pdef}
\end{equation}
The fractional Caffarelli-Kohn-Nirenberg inequality (CKN)  generalizes the classical CKN of \cite{CaffarelliKohnNirenberg} to fractional norms with homogeneous weights. It states that there exists a constant $\Lambda>0$ such that
\begin{equation}\label{ineq_u}
{\Lambda} \left(\int_{\r^n}\frac{|u(x)|^{p}}{|x|^{{\beta} {p}}}\,dx\right)^{\frac{2}{p}}\leq
\int_{\r^n}\int_{\r^n}\frac{(u(x)-u(y))^2}{|x-y|^{n+2\gamma}|x|^{{\alpha}}|y|^{{\alpha}}}\,dy\,dx
\end{equation}
for every $u\in C^{\infty}_c (\mathbb R^n)$.

 There has been extensive work regarding existence and properties of functions that satisfy equality above (we refer to them as {\it extremal solutions}),
see  for instance the  results in \cite{Frank-Lieb-Seiringer,Ghoussoub-Shakerian,Dipierro-Montoro-Peral-Sciunzi,Nguyen-Squassina,Musina-Nazarov1,Musina-Nazarov2,Chen-Lu-Zhang,Ao-DelaTorre-Gonzalez,deNitti-Glaudo-Konig} (this list is not  exhaustive).

Before we present our results let us introduce some notations. Consider   the energy functional
\begin{equation}\label{defi-E}
E_{\gamma,\alpha,\beta}[u]=\frac{\|u\|_{\gamma,\alpha}^2}{\Big(\int_{\r^n}|x|^{-\beta p}|u|^p\,dx\Big)^{2/p}},
\end{equation}
where we have defined the weighted norm
\begin{equation}\label{weighted-norm}
\|u\|_{\gamma,\alpha}^2:=
\int_{\r^n}\int_{\r^n}\frac{(u(x)-u(y))^2}{|x-y|^{n+2\gamma}|x|^{{\alpha}}|y|^{{\alpha}}}\,dy\,dx.
\end{equation}
We set
\begin{equation}\label{eq-S}
S_\gamma(\alpha,\beta):=\inf_{u\in D_{\gamma,\alpha}(\r^n)\setminus \{0\}}E_{\gamma,\alpha,\beta}[u]
\end{equation}
to be the best constant in  inequality \eqref{ineq_u}. Here the function space $D_{\gamma,\alpha}(\mathbb R^n)$ is defined in Section \ref{section:function-spaces} in terms of the norm \eqref{weighted-norm}.
We also consider the operator
\begin{equation}\label{formula-L}
\mathcal L_{\gamma,\alpha}u(x):=\textrm{PV}\int_{\r^n}\frac{u(x)-u(y)}{|x-y|^{n+2\gamma}|x|^{\alpha}|y|^{\alpha}}\,dy.
\end{equation}
Then extremal solutions for \eqref{eq-S}  are (weak) solutions of the  Euler-Lagrange equation
\begin{equation}
\mathcal L_{\gamma,\alpha}u(x)
=c\frac{|u(x)|^{p-2}u(x)}{|x|^{\beta p}}, \label{eq-extremal}
\end{equation}
for some constant $c$ which acts as a Lagrange multiplier. The existence and regularity of  minimizers of \eqref{ineq_u} and solutions to \eqref{eq-extremal}, has been addressed by several authors (see the references above). \\

% In Section we review those results and explore more closely how these results depend on the choice of parameters.
Note that  \eqref{ineq_u} can be understood as a critical fractional Hardy-Sobolev inequality (a good reference is the recent book \cite{Peral-Soria}). Indeed, by the change $\tilde u=|x|^{-\alpha}u$, it is equivalent to consider the energy functional (see Section \ref{section:equivalent} for details)
\begin{equation}\label{ground-state-representation}
\tilde E_{\gamma,\alpha,\beta}[\tilde u]=\frac{\displaystyle\int_{\mathbb R^n} \int_{\mathbb R^n} \frac{(\tilde u(x)-\tilde u(y))^2}{|x-y|^{n+2\gamma}}\,dx\,dy+ C_{\gamma,\alpha}\int_{\mathbb R^n} \frac{\tilde u^2(x)}{|x|^{2\gamma}} \, dx}
{\displaystyle \int_{\mathbb R^n} \frac{|\tilde u(x)|^p }{|x|^{(\beta-\alpha) p}}\,dx}.
\end{equation}
This equivalence is known as the \emph{ground state representation} of Frank-Lieb-Seiringer \cite{Frank-Lieb-Seiringer}. In particular,   $\mathcal L_{\gamma,\alpha}$ is just a (critical) Hardy type operator with fractional Laplacian
\begin{equation}
|x|^{\alpha}\mathcal L_{\gamma,\alpha}u= (\varsigma_\gamma)^{-1} (-\Delta)^\gamma \tilde u +
\frac{C_{\gamma,\alpha}}{|x|^{2\gamma}}\tilde u,
\label{rescaling L}\end{equation}
for some constants  $\varsigma_\gamma$ (given by \eqref{varsigma}) and $C_{\gamma,\alpha}$ defined precisely in Section \ref{section:equivalent}.
For simplicity, we define the new constant
$$c_{\text{CKN}}:=-\varsigma_\gamma C_{\gamma,\alpha}.$$
In  \cite{Ghoussoub-Shakerian} the authors prove existence of an extremal solution for $\tilde E_{\gamma,\alpha,\beta}$ as long as
\begin{equation}\label{range-existence}
c_{\text{CKN}}<c_{\text{H}},
\end{equation}
where $c_{\text{H}}$ is the Hardy constant, defined in \eqref{c-Hardy}.

 % since $C_{\gamma,\alpha}$ is a decreasing function of $\alpha$.
In addition,  a symmetrization argument from \cite{Ghoussoub-Shakerian} shows that $\tilde{u}$ is radial and decreasing if $ C_{\gamma,\alpha}<0$ (or equivalently, even  and decreasing for $x\geq 0$ if  $n=1$).
Another reference on the subject is \cite{Dipierro-Montoro-Peral-Sciunzi}, where they prove that $\tilde u\sim |x|^{-\alpha}$ as $|x|\to 0$. The separation of the region where solutions are radially symmetric from the symmetry breaking region, and the first properties of a Felli-Schneider type curve were considered in \cite{Ao-DelaTorre-Gonzalez}. \\

These results can be
 summarized by the following theorem (see Theorems 1.1 and 1.2 in \cite{Ghoussoub-Shakerian} and Proposition 4.7 in \cite{Ao-DelaTorre-Gonzalez}).

\begin{theorem}\label{theorem:properties}
Assume that $0<\gamma<1$, $0\leq p(\beta-\alpha)< 2 \gamma<n$ and
$c_{\text{CKN}}<c_{\text{H}}$. Then:
\begin{enumerate}
\item If either $\{\beta>\alpha\}$ or $\{\beta=\alpha\hbox{ and }   c_{\text{CKN}}\geq 0\}$  then the infimum of
$\tilde E_{\gamma,\alpha,\beta}[\cdot]$ is attained.
\item If $\beta=\alpha \hbox{ and }   c_{\text{CKN}}< 0$ then  there are no extremals for $\tilde E_{\gamma,\alpha,\beta}[\cdot]$ .
\item If $0< c_{\text{CKN}}$ or  $\{0=c_{\text{CKN}} \hbox{ and } 0<(\beta-\alpha)p<2 \gamma\}$ then any non-negative minimizer for $\tilde E_{\gamma,\alpha,\beta}[\cdot]$ is positive, radially symmetric, radially decreasing and approaches 0 as $|x|\to \infty$.
\item If the infimum of  $\tilde E_{\gamma,\alpha,\beta}[\cdot]$ is attained  then the extremal solution is strictly positive.
\end{enumerate}

\end{theorem}

In this work we are interested in the case $n=1$, more precisely
we will always assume that
$$n=1, \quad\gamma\in\big(0,\frac{1}{2}\big),\quad 0<\alpha<\frac{1-2\gamma}{2},\quad\alpha<\beta<\alpha+\gamma.$$

\begin{remark}\label{remark:intro}
The results described by Theorem \ref{theorem:properties}  are also valid for $n=1$. In addition, it is not difficult to verify that $0<c_{\text{CKN}}<c_{\text{H}}$ in our range of parameters
(see the discussion in Section \ref{section:equivalent}). Then, a minimizer exists, and it is positive, radially symmetric and decreasing in the radial variable.
\end{remark}

\begin{comment}
Results on existence and properties of  minimizers for \eqref{eq-S}, which are still valid in dimension $n=1$, can be found in \cite{Ao-DelaTorre-Gonzalez}. Let us summarize them in the following proposition.

\begin{proposition}\label{prop:minimizer}

One  has:
\begin{itemize}
\item For $\beta=\alpha+\gamma$, the best constant is not achieved.

\item For  $\alpha=\beta$, $0\leq \alpha<\frac{n-2\gamma}{2}$, $S(\alpha, \alpha)$ is achieved and the extremal solution is radially symmetric and decreasing in the radial variable. For $\alpha=\beta$, $-2\gamma<\alpha<0$, $S(\alpha, \alpha)$ is not achieved.

\item For $\alpha<\beta<\alpha+\gamma$, $S(\alpha,\beta)$ is always achieved. If, in addition, $\alpha\geq 0$, then the extremal solution is radially symmetric.
\end{itemize}
If an extremal solution exists, then it is strictly positive. \red{the meaning of strict positivity without continuity?}
\end{proposition}
\end{comment}

%Note that, even if we do not know if $u$ is radially decreasing, one can assume that $\displaystyle\max_{\overline{B_1}} u=1$ after rescaling (see the discussion in Section \ref{section:minimizers}).\\

The value $\gamma=\tfrac{1}{2}$ is critical for  \eqref{ineq_u} in dimension one  (see, for instance, the related reference \cite{Nguyen-Squassina}). In this paper we study this limiting case obtaining, in a suitable scale, a generalized Onofri's inequality.

Before we give the statement of our main theorem, let us  recall the work of  \cite{DET}, where they recover the classical Moser-Trudinger-Onofri's inequality as the liming case of the Caffarelli-Kohn-Nirenberg inequality  for the Laplace operator $(\gamma=1)$ in dimension $n=2$ as  $\alpha,\beta\to 0$ (see also the survey paper \cite{Dolbeault-Esteban-Jankowiak}). The classical CKN inequality in dimension $n=2$ reads
\begin{equation}\label{ineq-class}
{\Lambda_{\alpha,\beta}} \left(\int_{\r^2}\frac{|u(x)|^{p}}{|x|^{{\beta} {p}}}\,dx\right)^{\frac{2}{p}}\leq
\int_{\r^n}\frac{|\nabla u(x)|^2}{|x|^{{2\alpha}}}\,dx.
\end{equation}
There is extensive work on this inequality, including fundamental symmetry and symmetry breaking results (let us mention \cite{Felli-Schneider,Dolbeault-Esteban-Loss}, although by no means we are trying to be exhaustive).
Fix $-1<a< 0$ and choose the parameters
\begin{equation*}
p_\ve=\frac{2}{\varepsilon}, \quad \alpha_\varepsilon=-\frac{\varepsilon}{1-\varepsilon}(a+1),\quad \beta_\varepsilon=\alpha_\varepsilon+\varepsilon,
\quad\text{as}\quad \varepsilon\downarrow 0.
\end{equation*}
With this choice we  are in the symmetry region \cite[Theorem 1.2]{DET} so that the ``bubble" function
\begin{equation}\label{bubble}
u_\varepsilon(x):=\left(1+|x|^{-\frac{2\alpha_\ve(1+\alpha_\ve-\beta_\varepsilon)}{\beta_\ve-\alpha_\ve}}\right)
^{-\frac{\beta_\ve-\alpha_\ve}{1+\alpha_\ve-\beta_\ve}}
\end{equation}
is an extremal solution for \eqref{ineq-class}. Now define the measure
\begin{equation*}
d\nu_a=\frac{a+1}{\pi} \frac{|x|^{2a}}{(1+|x|^{2(a+1)})^2}\,dx
\end{equation*}
and set
\begin{equation*}
w_\ve=(1+\ve v)u_\ve\quad\text{for any}\quad v\in L^1(\mathbb R^2,d\nu_a),\text{ with }|\nabla v|\in L^2(\mathbb R^2,dx).
\end{equation*}
Then, substituting $w_\ve$ in the inequality \eqref{ineq-class}
 and passing to the limit as $\ve\downarrow 0$, one obtains
 %{\color{red}for $a>0$ we need $v$ to be radially symmetric. See Proposition 2.3 and the paragraph before it in \cite{DET}}
\begin{equation}\label{classical-Onofri}
 \log\int_{\mathbb R^2} e^v\,d\nu_a \leq \frac{1}{16 \pi (a+1)}\int_{\mathbb R^2} |\nabla v|^2\,dx+\int_{\mathbb R^2} v\,d\nu_a,
\end{equation}
This is a generalization of Onofri's inequality
\begin{equation*}
\int_{\mathbb R^2} e^{v-\int_{ \mathbb R^2} v\,d\nu_0}\,d\nu_0\leq e^{\frac{1}{16\pi}\int_{\mathbb R^2}|\nabla v|^2\,dx},
\end{equation*}
which corresponds to the classical Moser-Trudinger-Onofri inequality on $\mathbb S^2$ after stereographic projection.\\

Our main result recovers, at the limit, the ``boundary version" of  Onofri's  inequality on the unit disk  $D_1$  in $\mathbb R^2$, which reads as follows \cite{Osgood-Phillips-Sarnak}: for every $w\in W^{1,2}(D_1)$  it holds that
\begin{equation}\label{onofri}
\log\int_{\mathbb S^1} e^w \,\frac{d\theta}{2\pi}
\leq\frac14\int_{D_1}|\nabla w|^2 \,\frac{dx}{\pi}+\int_{\mathbb S^1}w\,\frac{d\theta}{2\pi}.
\end{equation}
This inequality  has appeared in many contexts. Indeed, it is the first Lebedev-Milin inequality \cite[Chapter 5]{Duren}, which is an important ingredient in the proof of Bieberach's conjecture (a historical account may be found in  \cite{Korevaar}). Also, improvements under constraints can be obtained from the work in \cite{Grenander-Szego} on Toeplitz
determinants (see the observation in \cite{Widom}).  From a more geometric point of view, the inequality was proven in \cite{Osgood-Phillips-Sarnak} in relation to extremals of functional determinants on manifolds with boundary (for which, in the simply connected case, the optimal is attained on the disk).

It is possible to rewrite \eqref{onofri} in terms of the half-Laplacian on $\mathbb R$ using stereographic projection. More precisely, for a given function $u\in H^\frac12(\pa D_1)$, we consider its harmonic extension in $D_1$, still denoted  by the same notation $u$. Then, recalling the definition of the half-Laplacian on $\mathbb S^1$ (see, for instance, \cite{DMR}), it holds that
$$\int_{D_1}|\nabla u|^2 \,dx=\int_{\mathbb S^1} u(-\Delta_{\mathbb S^1} )^\frac12 u \,d\theta.$$
After stereographic projection $\Pi:\mathbb S^1\to \mathbb R$, we can write \eqref{onofri} as
\begin{equation}\label{onofri-stereographic}
\log\frac{1}{2\pi}\int_{\mathbb R} e^v \,dm
\leq\frac{1}{4\pi}\int_{\mathbb R}v(-\Delta_{\mathbb R} )^{\frac{1}{2}} v\,dx+\frac{1}{2\pi}\int_{\mathbb R}v\,dm,
\end{equation}
where the measure is given by
\begin{equation*}
dm=\frac{2}{1+x^2}\,dx
\end{equation*}
and $v(x)=u(\Pi^{-1}(x))$. Here we have used that the half-Laplacian is a conformal operator in dimension one.
\\

As we approach the critical  exponent $\gamma \to (\frac{1}{2})^-$,  our main theorem shows that Onofri's inequality \eqref{onofri-stereographic} can be obtained as limit from the fractional Caffarelli-Kohn-Nirenberg inequality \eqref{ineq_u}  in dimension $n=1$:

\begin{theorem}\label{main-theorem} Let $n=1$ and choose a sequence of parameters
\begin{equation}\label{choice-parameters}
 \gamma_\ve\to \left(\tfrac{1}{2}\right)^-,\quad \beta_\ve\to0^+,\quad0<\alpha_\ve<\min\left\{\beta_\ve,  \frac{1-2\gamma_\ve}{2}\right\} , \quad \text{as} \quad\ve\downarrow 0,\quad\text{(in particular }p_\ve\to \infty),
 \end{equation}
%$$\gamma_\ve\to \left(\tfrac{1}{2}\right)^-,\quad \alpha_\ve \to 0, \quad \beta_\ve \to 0, \quad \text{as} \quad\ve\downarrow 0,\quad\text{(in particular }p_\ve\to \infty),$$
for which a (positive) radially symmetric minimizer $u_\ve$ exists.
Suppose that
\begin{equation*}
  \ve p_\ve\to a\in(0,\infty),\quad \beta_\ve p_\ve\to b \in[0,1).
  \end{equation*}  Up to a suitable normalization we assume that $u_\ve$ satisfies
  \begin{align}\label{eq-uepsilon-introduction}   \LL_\ve u_\ve=\frac{1}{p_\ve}  \frac{1}{|x|^{\beta_\ve p_\ve}} (u_\ve)^{p_\ve-1},\quad   \max_{\overline {B_1}}u_\ve =1.\end{align}
Then, the function
\begin{equation*}\label{eta-epsilon-introduction}
\eta_\ve:=p_\ve(u_\ve-1)
\end{equation*}
converges  (up to a subsequence) in  $C^0_{loc}(\R)$  to $ \eta_0$  as $\ve\downarrow 0$. Moreover, if we define
\begin{equation}
\label{shift}\eta:=\eta_0+\log\zeta_\frac12,
\end{equation}
then $\eta\in L^1_\frac 12(\R)$ (this space will be defined in \eqref{defL1q}) and $\eta$ is a solution to
\begin{equation}\label{eq-eta-introduction}(-\D)^\frac12\eta=|x|^{-b} e^\eta\quad\text{in }\R,\qquad \text{and}\qquad \kappa_b:= \int_\R |x|^{-b} e^\eta \,dx=2 \pi(1-b)<\infty.
\end{equation}
In addition, if we denote
\begin{equation*}
dm_b=\frac{e^\eta}{|x|^b}\,dx,
\end{equation*}
then we have
\begin{equation}\label{conclusion}
\log\left(\frac{1}{\kappa_b} \int_{\mathbb R} e^{v}\,dm_b\right)
\leq \frac{1}{2\kappa_b}  \int_{\mathbb R} v(-\Delta_{\mathbb R})^{\frac{1}{2}}v \,dx  +\frac{1}{\kappa_b} \int_{\mathbb R}  v\,dm_b, \end{equation}
for every $v\in C^\infty_0(\mathbb R)$.
\end{theorem}

%\textcolor{green}{I got a bit confused with the measures $d\nu_a$ and $dm_b$, maybe we need to write a bit more on $\eta$} {\color{red}The function $$\vp(x):=\log\frac{1}{1+|x|^{2(a+1)}}+\frac12\log\frac{a+1}{\pi}$$ satisfies $$-\Delta \vp=4\pi (a+1) |x|^{2a}e^{2\vp}=4\pi (a+1) |x|^{2a}\frac{1}{\left(1+|x|^{2(a+1)}\right)^2}\frac{a+1}{\pi},\quad \int_{\R^2}|x|^{2a}e^{2\vp}dx=1.$$ So, $\kappa=\int (-\Delta\vp)dx=4\pi(a+1)$, and  $dv_a=\frac{1}{\kappa} (-\Delta \vp)$}

\begin{remark} \label{rem:existenceofparameters}
It is always possible to choose a sequence of parameters in \eqref{choice-parameters} such that the hypotheses of Theorem \ref{main-theorem} hold.
For instance, by fixing $b\in [0,1)$, for $\ve>0$ small enough we can take
 $$a=1,\quad p_\ve:=\frac1\ve,\quad \beta_\ve=\ve(b+\ve),\quad \alpha_\ve:=\frac {\beta_\ve}{2},\quad \gamma_\ve:=\frac12-\ve\left[1-\frac12(b+\ve)\right].$$
 Indeed, this is a valid choice in the setting of Remark \ref{remark:intro}.

%\begin{equation*}
%\gamma_\ve=\frac{1}{2}-\frac{1-a}{a}\ve,\quad \alpha_\ve=\frac{b-a}{a}\ve,\quad \beta_\ve=\frac{b}{a}\ve,\quad p_\ve=\frac{a}{\ve}.
%\end{equation*}

%{\color{red}For the following choice of parameters do we have minimizers with those proerties?  Fix $m>>1$ such that $b(1+\frac1m)<1$, and for $\ve>0$ set $$p_\ve:=\frac1\ve,\quad \beta_\ve=\ve(b+\ve),\quad \alpha_\ve:=\frac {\beta_\ve}{m}$$}

%\textcolor{green}{these do not satisfy $\ve p_\ve\to a$. but may be not important}
\end{remark}

Our proof, although similar in spirit to that of  \cite{DET}, presents a challenging difficulty due to the fact that there is not an explicit formula for the extremal function such as the one in \eqref{bubble} for the classical case. Indeed, while the proof in \cite{DET} relies on explicit calculations as $\ve\to 0$ working with the function $w_\ve=(1+\ve v)u_\ve$  for $u_\ve$ defined in \eqref{bubble}, we need to find a different strategy to pass to the limit. Instead, we achieve compactness by establishing a-priori estimates with uniform constants that do not depend on the parameters.
\\

Problem \eqref{eq-eta-introduction}  has been well studied when $b=0$ (see \cite{Chang-Yang,Xu,DaLio-Martinazzi,DMR,Gerard-Lenzmann,Ahrend-Lenzmann}). In this particular case,  if $\eta\in L^1(dm_0)$ satisfies the non-local Liouville equation
$$(-\D)^\frac12\eta= e^\eta\quad\text{in }\R,\quad \int_\R e^\eta \,dx<\infty,$$
then we must have (up to translation and dilation)
$$\eta=\log\frac{2}{1+x^2}\quad \text{and in particular }\int_\R e^\eta \,dx=2\pi.$$
We observe that, as a consequence of our theorem for $b=0$  we recover Onofri's inequality \eqref{onofri-stereographic}.

In the presence of singular sources (the case $-1<b<1$), the classification of solutions is also known (see \cite{Galvez-Jimenez-Mira,Zhang-Zhou}). In fact, these are given by the one-parameter family
\begin{equation}
\eta_\rho(x)=\log\left(\frac{2(-b+1)\rho\cos\frac{\pi b}{2}}{|x|^{2(-b+1)}+2\rho |x|^{-b+1}\sin \frac{\pi b}{2}+\rho^2 }\right),\quad \rho>0. \label{solution-eta}
\end{equation}
In addition, non-degeneracy of this solution was studied in \cite{DelaTorre-Mancini-Pistoia}, while Brezis-Merle type concentration-compactness theorems were considered in \cite{Zhang-Zhou2}.

\begin{remark}\label{remark-nottrue}
For $b\leq -1$, there is no solution to the singular Liouville equation  \eqref{eq-eta-introduction} (and the function $\eta$ does not make sense in that context). However, we can show  (see the end of Section \ref{section:proof}) that the inequality \eqref{conclusion} does not hold for $-1<b<0$.

 \end{remark}

\begin{remark}
We expect that \eqref{classical-Onofri} also holds true if one considers the fractional CKN inequality in dimension $n=2$ and takes the limiting case $\gamma \to 1$; for this we refer to  the paper  \cite{Chang-Wang}, which is inspired in the dimension continuation argument of Branson \cite{Branson}. Indeed, continuation arguments are very well suited in the fractional world (see, for instance, the formal derivation in \cite[Section 6]{DelaTorre-Gonzalez-Hyder-Martinazzi} at the operator level).\\
\end{remark}

\begin{remark}\label{remark:parameters} The validity of \eqref{ineq_u} can be extended to a wider region of parameters
\begin{equation*}
\alpha\leq \beta\leq \alpha+\gamma, \quad \frac{n-2\gamma}{2}<\alpha<n,
\end{equation*}
(see the appendix  and the paper \cite{Nguyen-Squassina} for details); however, we will not consider those cases in our work.
\end{remark}

The structure of the paper is the following: in Section \ref{section:preliminaries} we give some background on the fractional Laplacian together with some preliminary results. Then, in Section \ref{section:minimizers} we consider the existence and properties of the extremal solutions to \eqref{ineq_u}).  Section \ref{section:compactness} yields the required compactness in order to pass to the limit. Finally, Section  \ref{section:proof} contains the proof of the main theorem. We conclude with an appendix proving Remark \ref{remark:parameters}.\\

\textbf{Acknowledgments:} M.d.M. Gonz\'alez  acknowledges financial support from grants PID2020-113596GB-I00, PID2023-150166NB-I00, RED2022-134784-T, all funded by MCIN/ AEI/10.13039/501100011033, and the ``Severo Ochoa Programme for Centers of Excellence in R\&D'' (CEX2019-000904-S). A. Hyder acknowledges  support from  SERB SRG/2022/001291.

\section{Preliminaries}\label{section:preliminaries}

\subsection{Function spaces}\label{section:function-spaces}

The fractional Laplacian on $\mathbb R^n$ is defined by
\begin{equation*}
(-\Delta)^\gamma u(x)=\varsigma_{\gamma}\textrm{PV}\int_{\mathbb R^n}\frac{u(x)-u(y)}{|x-y|^{n+2\gamma}}\,dy,
\end{equation*}
where
\begin{equation}
\varsigma_{\gamma}=\pi^{-\frac{n}{2}}2^{2\gamma}
\frac{\Gamma\left(\tfrac{n}{2}+\gamma\right)}{\Gamma(1-\gamma)}\gamma
\label{varsigma}
\end{equation}
is a multiplicative constant and $\textrm{PV}$ denotes the principal value.

We will consider  the space $L^1_q(\R^n)$  defined by \begin{equation}L^1_q(\R^n):=\left\{  u\in L^1_{loc}(\R^n):\|u\|_{L^1_q(\mathbb R^n)}<\infty\right\}, \label{defL1q}\end{equation}
where
$$\|u\|_{L^1_q(\mathbb R^n)}=\int_{\R^n}\frac{|u(x)|}{1+|x|^q}\,dx.$$

We now introduce a homogeneous weight. Let $D_{\gamma,\alpha}(\mathbb R^n)$ be the completion of $C^\infty_c(\mathbb R^n)$ with respect to the inner product
\begin{equation}\label{inner-product}
\langle u_1,u_2 \rangle_{\gamma,\alpha}:= \textrm{PV}
\int_{\r^n}\int_{\r^n}\frac{(u_1(x)-u_1(y))(u_2(x)-u_2(y))}{|x-y|^{n+2\gamma}|x|^{{\alpha}}|y|^{{\alpha}}}\,dy\,dx.
\end{equation}
Its associated norm is denoted by
\begin{equation*}\label{weighted-norm*}
\|u\|_{\gamma,\alpha}^2:=  \textrm{PV}
\int_{\r^n}\int_{\r^n}\frac{(u(x)-u(y))^2}{|x-y|^{n+2\gamma}|x|^{{\alpha}}|y|^{{\alpha}}}\,dy\,dx.
\end{equation*}
For simplicity, define the kernel
\begin{equation}
G_{\gamma,\alpha}(x,y):=\frac{1}{|x-y|^{n+2\gamma}|x|^\alpha |y|^\alpha},
\label{defiG}\end{equation}
so that
\begin{equation*}
\mathcal L_{\gamma,\alpha} u(x)= \textrm{PV}\int_{\r^n} [u(x)-u(y)]G_{\gamma,\alpha}(x,y)\,dy.
\end{equation*}
In what follows, all the singular integrals need to be interpreted in a principal value sense, but for the sake of simplifying the notation, we omit writing it explicitly.

Observe that formally we have
\begin{equation*}
\begin{split}
\|u\|_{\gamma,\alpha}^2&=\int_{\r^n}\int_{\r^n} [u(x)-u(y)]^2 G_{\gamma,\alpha}(x,y)\,dxdy\\
&=\int_{\r^n}\int_{\r^n} [u(x)-u(y)]u(x)G_{\gamma,\alpha}(x,y)\,dxdy+\int_{\r^n}\int_{\r^n} [-u(x)+u(y)]u(y)G_{\gamma,\alpha}(x,y)\,dy\\
&=2\int_{\r^n}\int_{\r^n} u(x)\mathcal L_{\gamma,\alpha} u(x)\,dx.
\end{split}
\end{equation*}
%The computation can be made rigorous when considered in the principal value sense.

Now let us consider the equation
\begin{equation}\label{eq}
\mathcal L_{\gamma,\alpha} u=F\quad\text{in }\mathbb R^n.
\end{equation}
Given $F\in L^2(\mathbb R^n)$, we say that $u\in D_{\gamma,\alpha}$ is a weak energy solution of \eqref{eq} if
\begin{equation}\label{weak-formulation}
\frac{1}{2}\int_{\r^n}\int_{\r^n}\frac{(u(x)-u(y))(\phi(x)-\phi (y))}{|x-y|^{n+2\gamma}|x|^{{\alpha}}|y|^{{\alpha}}}\,dy\,dx=\int_{\mathbb R^n} F(x)\phi(x) \,dx\quad\text{for all }\phi\in C^{\infty}_c(\mathbb R^n).
\end{equation}

%Given  $f\in ?$, we say that $u\in L^1_\gamma$ is a weak solution  of \eqref{eq} if \textcolor{green}{to finish later}

\subsection{An equivalent formulation} \label{section:equivalent}
Here we give the ground state representation \eqref{ground-state-representation} for our problem.

Note that $$|x|^{-\alpha}|y|^{-\alpha}(u(x)-u(y))=|x|^{-\alpha}(|x|^{-\alpha}u(x)-|y|^{-\alpha}u(y))+ |x|^{-\alpha} u(x)(|y|^{-\alpha}-|x|^{-\alpha}),$$
then, writing $\tilde u=|x|^{-\alpha}u$,  from formula \eqref{formula-L} we obtain
$$
\mathcal L_{\gamma,\alpha}u=|x|^{-\alpha} \int_{\mathbb R^n} \frac{\tilde u(x)-\tilde u(y)}{|x-y|^{n+2\gamma}}\,dy+
C_{\gamma,\alpha}|x|^{-\alpha-2\gamma}\tilde u(x),$$
where
\begin{equation*}
C_{\gamma,\alpha}=|x|^{\alpha+2\gamma}  \textrm{PV}\int_{\mathbb R^n} \frac{|y|^{-\alpha}-|x|^{-\alpha}}{|x-y|^{n+2\gamma}}\,dy.
\end{equation*}
Changing variables $z=\frac{y}{|x|}$ we arrive to
\begin{equation}\label{constant-C}C_{\gamma,\alpha}=  \int_{\mathbb R^n} \frac{|z|^{-\alpha}-1}{|e_1-z|^{n+2\gamma}}\,dz.
\end{equation}
Trivially, $C_{\gamma,0}=0$. A simple computation shows that this constant has a sign. More precisely:

\begin{lemma}\label{lemma:constant} In our range of parameters, $C_{\gamma,\alpha}$ is a decreasing function of $\alpha$. If in addition we assume
 $0<\alpha<n-2\gamma$, then we have that $ C_{\gamma,\alpha}<0$.
\end{lemma}

\begin{proof}
With the change of variables  $z\mapsto \frac{z}{|z|^2}$ in \eqref{constant-C} for the integral on $B_1^c$, and using the identity \begin{equation}\label{Kelvin}
|x||y|\left|\frac{x}{|x|^2}-\frac{y}{|y|^2}\right|=|x-y|
\end{equation}
one obtains
\begin{align*}  C_{\gamma,\alpha}&= \left( \int_{B_1}+\int_{B_1^c}\right) \frac{|z|^{-\alpha}-1}{|e_1-z|^{n+2\gamma}}\,dz\\ &= \int_{ B_1}  \frac{|z|^{-\alpha}-1}{|e_1-z|^{n+2\gamma}}\,dz+ \int_{B_1}  \frac{|z|^{\alpha}-1}{|e_1-z|^{n+2\gamma}}|z|^{-n+2\gamma}\,dz  \\ &=  \int_{B_1}  \frac{(|z|^{-\alpha}-1)(1-|z|^{-n+2\gamma+\alpha})}{|e_1-z|^{n+2\gamma}}\,dz.
\end{align*}
It follows that the  sign of $C_{\gamma,\alpha}$ depends on the sign of $\alpha$ and $-n+2\gamma+\alpha$. In the particular case, $0<\alpha<n-2\gamma$, we clearly have $C_{\gamma,\alpha}<0$ as the numerator is negative in $B_1$.
Moreover, $C_{\gamma,\alpha} $ is decreasing for $\alpha\leq\frac{n-2\gamma}{2}$, and increasing for $\alpha\geq \frac{n-2\gamma}{2}$. To see this, for a fixed $0<|z|<1$ we define $$f(\alpha):=(|z|^{-\alpha}-1)(1-|z|^{-n+2\gamma+\alpha}).$$ Then we have \begin{align*} f'(\alpha)&=-|z|^{-\alpha}\log|z|\left(  1-|z|^{-n+2\gamma+2\alpha}\right) \left\{ \begin{array}{ll} >0 &\quad\text{if }\alpha>\frac{n-2\gamma}{2} \\ <0&\quad\text{if }\alpha<\frac{n-2\gamma}{2} ,\end{array} \right.\\ \end{align*}  giving  the desired  monotonicity of $C_{\gamma,\alpha}$ with respect to $\alpha$.
\end{proof}

Note that for the particular value $\alpha=\frac{n-2}{2}$ we reduce to the Hardy constant
(see, for instance, the appendix in \cite{Ao-DelaTorre-Gonzalez}). More precisely,
\begin{equation}\label{c-Hardy}
c_{\text{H}}:=\varsigma_\gamma \int_{\mathbb R^n} \frac{1-|z|^{-\frac{n-2\gamma}{2}}}{|e_1-z|^{n+2\gamma}}\,dz= 2^{2\gamma}\frac{\Gamma^2\left(\frac{n+2\gamma}{4}\right)}{\Gamma^2\left(\frac{n-2\gamma}{4}\right)}=-\varsigma_\gamma C_{\gamma,\frac{n-2\gamma}{2}}.
\end{equation}

 With a similar computation we have
$$
\|u\|_{\gamma,\alpha}^2:=
\int_{\mathbb R^n}\int_{\mathbb R^n}\frac{(\tilde u(x)-\tilde u(y))^2}{|x-y|^{n+2\gamma}}\,dy\,dx
+ C_{\gamma,\alpha}\int_{\mathbb R^n} \frac{\tilde u^2(x)}{|x|^{2\gamma}} \, dx, \quad \text{where}\quad \tilde u=|x|^{-\alpha}u. $$
Thus it is possible to consider the equivalent functional
\begin{equation*}
\tilde E_{\gamma,\alpha,\beta}[\tilde u]=\frac{\displaystyle\int_{\mathbb R^n} \int_{\mathbb R^n} \frac{(\tilde u(x)-\tilde u(y))^2}{|x-y|^{n+2\gamma}}\,dx\,dy+ C_{\gamma,\alpha}\int_{\mathbb R^n} \frac{\tilde u^2(x)}{|x|^{2\gamma}} \, dx}
{\displaystyle \int_{\mathbb R^n} |\tilde u(x)|^p |x|^{(-\beta+\alpha) p}\,dx}.
\end{equation*}
The existence and properties of minimizers  of $\tilde E_{\gamma,\alpha,\beta}$ were already partially discussed in the introductory section above. In addition,  further regularity properties will be presented in
Section \ref{section:minimizers}.

The operator  \eqref{formula-L} in the new notation has the form
\begin{equation*}
\tilde {\mathcal L}_{\gamma,\alpha} (\tilde u)= (\varsigma_\gamma)^{-1} (-\Delta)^\gamma \tilde u+C_{\gamma,\alpha}\frac{1}{|x|^{2\gamma}}\tilde u.
\end{equation*}

\subsection{Operation rules}

\begin{lemma}(Integration by parts)
Let $u,v\in D_{\gamma,\alpha}$. It holds
\begin{equation}\label{integration-by-parts}
\int_{\mathbb R^n} v(x)\mathcal L_{\gamma,\alpha}u(x)\,dx=\int_{\mathbb R^n} u(x)\mathcal L_{\gamma,\alpha}v(x)\,dx.
\end{equation}
\end{lemma}
\begin{proof}
Calculate
\begin{equation*}
\begin{split}
I:&=\int_{\mathbb R^n} v(x)\mathcal L_{\gamma,\alpha}u(x)\,dx=\int_{\mathbb R^n}\int_{\mathbb R^n} [u(x)-u(y)]v(x)G_{\gamma,\alpha}(x,y)\,dxdy\\
&=\int_{\mathbb R^n}\int_{\mathbb R^n} u(x)[v(x)-v(y)]G_{\gamma,\alpha}(x,y)\,dxdy+\int_{\mathbb R^n}\int_{\mathbb R^n} u(y)[v(y)-v(x)]G_{\gamma,\alpha}(x,y)\,dxdy\\
&+\int_{\mathbb R^n}\int_{\mathbb R^n} [u(x)-u(y)]v(y)G_{\gamma,\alpha}(x,y)\,dxdy\\
&=2\int_{\mathbb R^n} u(x)\mathcal L_{\gamma,\alpha} v(x)\,dx-I,
\end{split}
\end{equation*}
from where we obtain our claim.
\end{proof}

\begin{lemma}(Product formula) Let $u,v\in D_{\gamma,\alpha}$. Then

\begin{equation}\label{product-formula}
\mathcal L_{\gamma,\alpha}(uv)=v\mathcal L_{\gamma,\alpha}u+ u\mathcal L_{\gamma,\alpha}v-\int_{\mathbb R^n} [u(\cdot)-u(y)][v(\cdot)-v(y)]G_{\gamma,\alpha}(\cdot,y)\,dy.
\end{equation}

\end{lemma}

\begin{proof}
We expand
\begin{equation*}
u(x)v(x)-u(y)v(y)=[u(x)-u(y)]v(x)+[v(x)-v(y)]u(x)-[u(x)-u(y)][v(x)-v(y)],
\end{equation*}
and the lemma follows immediately.
\end{proof}

\section{ Regularity of minimizers}\label{section:minimizers}

In this section we focus on the regularity of minimizers  of  \eqref{ineq_u}. Although existence was already discussed in Theorem \ref{theorem:properties}, in order to prove our main result we need further regularity properties.  Throughout this section we work with bounded solutions to Equation \eqref{PDE} that additionally satisfy an integral bound.
We show the following regularity result (independent of the parameters).

\begin{theorem} \label{regularity} Let $\theta\in[0,\theta_0)$ for some $0<\theta_0<1$ fixed. Let $u$ be a locally  bounded weak solution of \begin{align}\label{PDE}\LL_{\gamma,\alpha} u=\frac{f}{|x|^{\theta}}\quad\text{in }(-2M,2M), \end{align} for some $M>>1$,   $\alpha\in \left(0,\, \frac12\right)$. Assume, in addition, that $f$ is uniformly bounded in
$(-2M,2M),$ and
$$\|u\|_{L^1_q(\mathbb R)}<\infty, \quad\hbox{with}\quad  q:=1+2\gamma+\alpha.$$
Then there exist  $\gamma_0\in \left(0,\frac12\right)$  (possibly close to $\frac 1 2$) and  $\sigma>0$ (only depending on $\gamma_0,\theta_0$ but not on $\alpha$) such that for  $\gamma_0\leq\gamma<\frac12$ we have
$$\sup_{x,y\in B_\frac M2}\frac{|u(x)-u(y)|}{|x-y|^\sigma}\leq C,$$
where $C$ is a constant that depends on $M$, $ \|u\|_{L^\infty(B_M)} $,  $\|f\|_{L^\infty(B_M)}$ and
$\|u\|_{L^1_q(\mathbb R)}$.
\end{theorem}

The proof of Theorem \ref{regularity} is based on a regularity argument developed in \cite{Schikorra} for fractional harmonic maps, using the hole-filling trick  introduced by  Widman in \cite{Widman}. We divide this section into two subsections,  the first one (Subsection \ref{section:lemmas}) contains several technical lemmas that are necessary to set up a Campanato iteration. We will show, in particular, that for the seminorm \eqref{seminorm} defined below it holds
\begin{equation*}
[u]_\rho^2\leq\tau [u]^2_{2^{k_0}\rho}+ \textrm{error terms},
\end{equation*}
for some $\tau<1$, $k_0\in\mathbb N$.  The error terms will be explicit in Proposition \ref{Prop-C2}.
We conclude  by proving Theorem \ref{regularity} in Subsection \ref{section:pfregthem}.\\

Throughout this section we use repeatedly the  notation that we collect in the following lines.

\begin{itemize}
\item {\it Weighted measure:} \begin{equation}\label{measure}
|S|_{\mu_\alpha}=\int_S d\mu_\alpha(x),\quad \text{where}\quad  d\mu_\alpha(x):=\frac{dx}{|x|^\alpha}.
\end{equation}
\item {\it Annulus around a point: }
\begin{equation}\label{annuli}A_{\rho,R}(x_0):=[B_R(x_0)\times(B_R(x_0)\setminus B_\rho(x_0))]\cup [(B_R(x_0)\setminus B_\rho(x_0))\times B_R(x_0)]. \end{equation}
\item {\it Localized semi-norm associated to $u\in D_{\gamma,\alpha}(\mathbb R)$}
\begin{equation}\label{seminorm}
[u]^2_{S}:=\int_{S}\int_{S}G_{\gamma,\alpha}(x,y)|u(x)-u(y)|^2\,dxdy.
\end{equation}
If $S=B_r(x_0)$, we denote $[u]^2_{S}=[u]^2_{r}$. %{\color{blue} We are not really using this notation, except in the statement}

\item {\it Averages of $u$:} Fix $R>0$ and  $0<\rho<R$. Given a function $u$ we denote
\begin{align*} \bar u_\rho &:=\fint_{B_\rho(x_0)} u (x)\,d\mu_\alpha(x),\\
 \hat u_\rho &:=\fint_{B_R(x_0)\setminus B_\rho(x_0)}u(x)\,d\mu_\alpha (x). \end{align*}

\item {\it Localization of $u$:}
 Fix $\rho>0$ and let $\vp\in C_c^\infty(B_{2\rho}(x_0))$ be such that $\vp\geq 0$, $\vp\equiv 1$ in $B_\rho(x_0)$ and $|\nabla \vp|\leq \frac{2}{\rho}$. Given a function $u$ we set
\begin{equation}
 \psi_u:=(u-\bar u_\rho)\vp. \label{defpsi}\end{equation}
 In the context below the function $u$ will be fixed (as solution to Equation \eqref{PDE}) and we omit the subindex of $\psi$.

\item{\it The tail of a function:} Fix a function $u$, $x_0\in\mathbb R^n$ and $R>0$. Let
 $$\Tail_\alpha(u;x_0,R):=\left[R^{2\gamma +\alpha }\int_{\mathbb{R}^n\setminus B_R(x_0)} |u(x)||x-x_0|^{-n-2\gamma}\, d\mu_\alpha(x) \right].$$

\end{itemize}
\subsection{ Technical lemmas}\label{section:lemmas}

We first state two results that can be obtained by a straightforward computation.

\begin{lemma}\label{lem-C5} There exists $C>0$ such that for every $x_0\in\R$, $r>0$ and $0\leq\alpha\leq\frac12$ we have
\begin{align}
\frac1C|B_r(x_0)|_{\mu_\alpha}\leq\left\{\begin{array}{ll} r^{1-\alpha}&\quad\text{if }|x_0|\leq 2r\\ r|x_0|^{-\alpha}&\quad\text{if }|x_0|\geq2r\end{array}\right\} \leq C|B_r(x_0)|_{\mu_\alpha}.\end{align}\end{lemma}

\begin{lemma}\label{lem-C3}
We have, for  $0<\rho<R$,
\begin{align} \int_{A_{\rho,R}}G_{\gamma,\alpha}(x,y)|u(x)-u(y)|^2\,dxdy\leq 2\left( [u]^2_{R}-[u]_{\rho}^2\right). \end{align}
 \end{lemma}

In what follows we assume that $u$ is a locally bounded  weak solution to Equation \eqref{PDE}. We also fix  $\rho>0$ and a point  $x_0\in\mathbb{R}^n$  that we assume (in what remains of this section) is the center for  all balls and annuli, but we omit writing it  explicitly to simplify the notation.
%In what remains of this section we fix
%$\rho>0$ and consider the definitions above

The main goal of this subsection is to prove the following proposition, which contains the main idea in the hole filling technique.

%Caccioppoli inequality: For $k\in\R$ we set $w:=u-k$, $w_+:=\max\{w,0\}$, $w_-:=\max\{-w,0\}$.

\begin{proposition}\label{Prop-C2}Let  $u$ be a solution of  \eqref{PDE} and $M>>1$ be fixed. For $k_0\geq 3$ there  exist constants $C_0$ and  $C(k_0)>1$ such that for every $0\leq\alpha\leq\frac12$, $\frac13\leq\gamma<\frac12$, $0\leq|x_0|\leq\frac M2$ we have
\begin{align}[u]_\rho^2&\leq\tau [u]^2_{2^{k_0}\rho}+\frac{C_0}{1+C(k_0)}\sum_{k=k_0+1}^{N+1}2^{-k\gamma}[u]_{2^k\rho}^2\\&\quad +C_0\rho^{1-\alpha}\|u\|_{L^\infty(B_{2\rho})}\left( \|u\|_{L^\infty(B_{2\rho})}+\Tail(u,x_0,1)\right)+C_0\rho^{1-\theta}\|u\|_{L^\infty(B_{2\rho})}\|f\|_{L^\infty(B_{2\rho})},\end{align}
 where  $ \tau:= \frac{C(k_0)}{1+C(k_0)}<1$, $\theta$ is   given by \eqref{PDE}, $\rho>0$ and $N\geq k_0+1$ satisfies  $1\leq 2^{N+1}\rho\leq 2 $. The constant $C_0$ is independent of $\rho, k_0, N,\alpha,\gamma$. %, and for $r>0$ $$[u]_r^2:=[u]^2_{B_r}:=\int_{B_r}\int_{B_r}G_{\gamma,\alpha}(x,y)|u(x)-u(y)|^2\,dxdy,\quad B_r:=B_r(x_0).$$

\end{proposition}

In what remains of this section we need to keep track on the dependence of $k_0$ of the constants,  hence,  the constants that depend on $k_0$ will be denoted by $C_{k_0}$, while $C$ will denote universal constants (that do not depend on the choice of $k_0$).

Before giving the proof of  Proposition \ref{Prop-C2} we need two intermediate lemmas to handle the lower order terms.

\begin{lemma}
\label{lem-C4} Let $u$ and the parameters $\gamma$, $\alpha$  be as in Proposition \ref{Prop-C2}. Fix $R= 2^{k_0}\rho$ for some $\rho>0$, $k_0\geq 3$. Then it holds
\begin{align} \int_{B_1\setminus B_R}\int_{B_{ R}}G_{\gamma,\alpha}(x,y)|u(x)-u(y)||\psi(x)|\,dxdy\leq C\sum_{k=k_0}^N 2^{-k\gamma}[u]^2_{{2^{k+1}\rho}}, \end{align}
where $N\in\mathbb N$ satisfies  $1\leq 2^{N+1}\rho\leq 2 $, and the constant $C$ is independent of $\rho, k_0, N,\alpha,\gamma,u$, $M$.  \end{lemma}

\begin{proof}
Note that for suitable $k_0$ and $N$ we have $B_1\setminus B_R\subset \displaystyle\cup_{k=k_0}^N(B_{2^{k+1}\rho}\setminus B_{2^k\rho})$. Since we are assuming that $\textrm{supp}(\vp)\subset B_{2\rho}$, the definition of $\psi$ given by \eqref{defpsi} implies
 \begin{align}  \int_{B_1\setminus B_R}\int_{B_R}&G_{\gamma,\alpha}(x,y)|u(x)-u(y)||\psi(x)|\,dxdy\\&\leq \sum_{k=k_0}^N\int_{B_{2^{k+1}\rho}\setminus B_{2^k\rho}} \int_{B_{2\rho}}G_{\gamma,\alpha}(x,y)|u(x)-u(y)||u(x)-\bar u_\rho|\,dxdy.
 \end{align}
% The last inequality follows from the definition of $\psi$ given by \eqref{defpsi}.

  By H\"older's inequality
  \begin{align} \int_{B_{2^{k+1}\rho}\setminus B_{2^k\rho}} \int_{B_{2\rho}}&G_{\gamma,\alpha}(x,y)|u(x)-u(y)||u(x)-\bar u_\rho|\,dxd y\\& \leq [u]_{{2^{k+1}\rho}}\left( \int_{B_{2^{k+1}\rho}\setminus B_{2^k\rho}} \int_{B_{2\rho}}G_{\gamma,\alpha}(x,y)|u(x)-\bar u_\rho|^2\,dxd y\right) ^\frac12. \end{align}
   Since $k\geq k_0\geq3$, for $x\in B_{2\rho}$ and $y\in B_{2^{k+1}\rho}\setminus B_{2^k\rho}$ we have that $|x-y|\approx 2^k\rho$. Therefore, by Jensen's inequality \begin{align}
    \int_{B_{2^{k+1}\rho}\setminus B_{2^k\rho}} \int_{B_{2\rho}}G_{\gamma,\alpha}(x,y)|u(x)-\bar u_\rho|^2\,dxd y \leq &C(2^k\rho)^{-1-2\gamma}\frac{|B_{2^{k+1}\rho}|_{\mu_\alpha}}{|B_\rho|_{\mu_\alpha}}\int_{B_{2\rho}}\int_{B_\rho}|u(x)-u(z)|^2\,d\mu_\alpha(z)d\mu_\alpha(x)\\&\leq C2^{-k(1+2\gamma)} \frac{|B_{2^{k+1}\rho}|_{\mu_\alpha}}{|B_\rho|_{\mu_\alpha}}\int_{B_{2\rho}}  \int_{B_\rho}G_{\gamma,\alpha}(x,z)|u(x)-u(z)|^2\,d zdx\\&\leq C 2^{-2k\gamma}[u]^2_{{2\rho}},\end{align}
    where the last inequality follows from  Lemma \ref{lem-C5} since
    $$\frac{|B_{2^{k+1}\rho}|_{\mu_\alpha}}{|B_\rho|_{\mu_\alpha}}\leq C 2^k.$$

\end{proof}

\begin{lemma} \label{lem-C6} Consider  $R\geq 2\rho$ and the functions $u$ and $
\psi$ as previously defined. Then it holds
\begin{align*}
\int_{\R\setminus B_1}\int_{B_R}G_{\gamma,\alpha}(x,y)|u(x)-u(y)||\psi(x)|\,dxdy\leq   C\rho^{1-\alpha}\|u\|_{L^\infty(B_{2\rho})}\left( \|u\|_{L^\infty(B_{2\rho})}+\Tail(u,x_0,1)\right). \end{align*}
\end{lemma}
\begin{proof} We have
\begin{align}
\int_{\R\setminus B_1}\int_{B_R}&G_{\gamma,\alpha}(x,y)|u(x)-u(y)||\psi(x)|\,dxdy\\& \leq  C\|u\|_{L^\infty(B_{2\rho})} |B_{2\rho}|_{\mu_\alpha} \int_{B^c_1}\frac{\|u\|_{L^\infty(B_{2\rho})}+|u(y)|}{|x_0-y|^{1+2\gamma}}\,\frac{dy}{|y|^\alpha}\\&\leq C\rho^{1-\alpha}\|u\|_{L^\infty(B_{2\rho})}\left( \|u\|_{L^\infty(B_{2\rho})}+\Tail(u,x_0,1)\right).
\end{align}
Note that here we have used that, on this domain, $|x-y|\geq\frac12 |x_0-y|$ since  $\psi$ is supported in $B_{2\rho}$.
\end{proof}

We now proceed with the proof of the main result in this subsection.

\begin{proof}[Proof of Proposition \ref{Prop-C2}]
In the notation above, we set $\psi$ as introduced in \eqref{defpsi} and $R= 2^{k_0}\rho$.
%Let $\vp\in C_c^\infty(B_{2\rho}(x_0))$ be such that $\vp\equiv 1$ in $B_\rho(x_0)$, $\vp\in [0,1]$ and $|\nabla \vp|\leq \frac{2}{\rho}$, as above. We also set  $ \psi=(u-\bar u_\rho)\vp$.
 Writing
 \begin{equation}\label{note1}
 \psi(x)-\psi(y)=[u(x)-u(y)]-[1-\vp(x)][u(x)-u(y)]+[\vp(x)-\vp(y)][u(y)-\bar u_\rho],\end{equation}
 and using that $R> \rho$ and $\vp\equiv1$ in $B_\rho$ we obtain
 \begin{equation}\label{eq:hole}
 \begin{split}
 \int_{B_\rho} \int_{B_\rho}&G_{\gamma,\alpha}(x,y)| u(x)-u(y)|^2\,dxdy\\
 &\leq \int_{B_R} \int_{B_R} G_{\gamma,\alpha}(x,y)| \psi(x)-\psi(y)|^2\,dxdy \\
 &=\int_{B_R} \int_{B_R}G_{\gamma,\alpha}(x,y)[u(x)-u(y)][ \psi(x)-\psi(y)]\,dxdy  \\
 &\quad -\int_{B_R} \int_{B_R}G_{\gamma,\alpha}(x,y)[1-\vp(x)][u(x)-u(y)][ \psi(x)-\psi(y)]\,dxdy\\
 &\quad +\int_{B_R} \int_{B_R}G_{\gamma,\alpha}(x,y)[\vp(x)-\vp(y)][u(y)-\bar u_\rho][ \psi(x)-\psi(y)]\,dxdy\\&=:I_1+I_2+I_3.
 \end{split}
 \end{equation}
 In order to estimate $I_2$ we note that the integrand vanishes for $x\in B_\rho$, and by \eqref{note1} and the  Cauchy-Schwartz inequality,% for any $\ve>0$
\begin{equation*}
  \begin{split}
 |u(x)-u(y)||\psi(x)-\psi(y)|&\leq |u(x)-u(y)|\left\{|u(x)-u(y)|+|\vp(x)-\vp(y)||u(y)-\bar u_\rho|\right\}\\&\leq 2|u(x)-u(y)|^2+ 2|\vp(x)-\vp(y)|^2|u(y)-\bar u_\rho|^2.
\end{split}
  \end{equation*}
  In a similar way, the integrand in $I_3$ vanishes if both $x$ and $y$ are in $B_\rho$, and
  \begin{align}| \psi(x)-\psi(y)||\vp(x)-\vp(y)||u(y)-\bar u_\rho|\leq 2|u(x)-u(y)|^2+ 2|\vp(x)-\vp(y)|^2|u(y)-\bar u_\rho|^2. \end{align}
 % Thus, setting $A_{\rho,R}:=[B_R\times(B_R\setminus B_\rho)]\cup [(B_R\setminus B_\rho)\times B_R]$ (all balls are centered at $x_0$)
 Recalling the definition of the annulus \eqref{annuli} and   Lemma \ref{lem-C3} we have
   \begin{align}
    |I_2|+|I_3|&\leq 4 \int _{A_{\rho,R}} G_{\gamma,\alpha}(x,y)\left(|u(x)-u(y)|^2+ |\vp(x)-\vp(y)|^2|u(y)-\bar u_\rho|^2\right)\,dxdy\\&\leq 8\left( [u]^2_{R}-[u]_{\rho}^2\right)+4 \int _{A_{\rho,R}} G_{\gamma,\alpha}(x,y)  |\vp(x)-\vp(y)|^2|u(y)-\bar u_\rho|^2\,dxdy.  \label{bdi23}
    \end{align}
 In order to estimate this last term  we denote
   \begin{equation}I_4:= \int _{A_{\rho,R}} G_{\gamma,\alpha}(x,y)  |\vp(x)-\vp(y)|^2|u(y)-\bar u_\rho|^2\,dxdy.
\end{equation}
Let us estimate this term $I_4$. First, we observe by a direct computation (which involves expanding the integrand below) that% Note that by expanding the integrand we have that
 \begin{equation*}\begin{split} |B_R\setminus B_\rho|^{-1}_{\mu_\alpha}\,|B_\rho|^{-1}_{\mu_\alpha}&\int_{B_R\setminus B_\rho}\int_{B_\rho}|u(x)-u(z)|^2\,d\mu_\alpha(x)\,d\mu_\alpha(z)\\ &= \fint_{ B_\rho}u^2(x)\,d\mu_\alpha (x)
    +\fint_{B_R\setminus B_\rho}u^2(z)\,d\mu_\alpha (z)-2 \hat u_\rho\,\bar u_\rho.\end{split}\end{equation*}
Second, we will also need the following bound
    %Then, using H\"older's inequality, the fact  that  $|z-y|\leq R$ for $z\in  B_\rho$, $y \in B_R\setminus B_\rho$, $\gamma\leq 1$,
    %and Lemma  \ref{lem-C3} we have
       \begin{align} |\hat u_\rho-\bar u_\rho|^2&\leq |B_R\setminus B_\rho|^{-1}_{\mu_\alpha}\,|B_\rho|^{-1}_{\mu_\alpha}\int_{B_R\setminus B_\rho}\int_{B_\rho}|u(x)-u(z)|^2\,d\mu_\alpha(x)\,d\mu_\alpha(z)\\&\leq  |B_R\setminus B_\rho|^{-1}_{\mu_\alpha}\,|B_\rho|^{-1}_{\mu_\alpha} (2R)^{1+2\gamma}\int_{B_R\setminus B_\rho}\int_{B_\rho}G_{\gamma,\alpha}(x,z)|u(x)-u(z)|^2\,dx dz\\
    &\leq 8R^{1+2\gamma}|B_R\setminus B_\rho|^{-1}_{\mu_\alpha}\,|B_\rho|^{-1}_{\mu_\alpha} \left( [u]^2_{R}-[u]_{\rho}^2\right). \label{est-propo-1}% \\&\blu{ \leq C_{k_0}\left( [u]^2_{B_R}-[u]_{B_\rho}^2\right) ??}
    \end{align}
Note that     in the previous computation we used H\"older's inequality, the fact  that  $|z-y|\leq 2R$ for $z\in  B_\rho$, $y \in B_R\setminus 2B_\rho$, $\gamma\leq 1$,
    and Lemma  \ref{lem-C3}.

Third,  we observe from the gradient bound of $\vp$, and taking into account that  $|x-y|\leq 2R$ for $x\in  B_\rho$, $y \in B_R\setminus B_\rho$ that it holds
    \begin{equation*}
     \begin{split}
     \int_{A_{\rho,R}}G_{\gamma,\alpha}(x,y)&|\vp(x)-\vp(y)|^2|u(y)-\hat u_\rho|^2\,dxdy\\
    &\leq R^{1-2\gamma}\frac{4}{\rho^2}\int_{A_{\rho,R}} |u(y)-\hat u_\rho|^2\,d\mu_\alpha(x)d\mu_\alpha(y)\\&\leq CR^{1-2\gamma}\frac{|B_R|_{\mu_\alpha}}{\rho^2}\int_{B_R} |u(y)-\hat u_\rho|^2\,d\mu_\alpha(y)\\
   &=   CR^{1-2\gamma}\frac{|B_R|_{\mu_\alpha}}{\rho^2}\int_{B_R} \left|
   |B_R\setminus B_\rho|_{\mu_\alpha}^{-1}\int_{B_R\setminus B_\rho} (u(y)- u(z)) d\mu_\alpha(z) \right|^2\,d\mu_\alpha(y)\\
     &\leq CR^{1-2\gamma}\frac{|B_R|_{\mu_\alpha}}{\rho^2|B_R\setminus B_\rho|_{\mu_\alpha}}\int_{B_R}\int_{B_R\setminus B_\rho} |u(y)- u(z)|^2\,d\mu_\alpha(z)d\mu_\alpha(y) \\&\leq CR^{2}\frac{|B_R|_{\mu_\alpha}}{\rho^2|B_R\setminus B_\rho|_{\mu_\alpha}}\int_{B_R}\int_{B_R\setminus B_\rho} G_{\gamma,\alpha}(y,z)|u(y)- u(z)|^2\,dzd y.
     %\\&\blu{ \leq C_{k_0}\left( [u]^2_{B_R}-[u]_{B_\rho}^2\right) ??}
    % \label{est-propo-2}
     \end{split}
      \end{equation*}
From Lemma \ref{lem-C5}  and our assumption that $\frac{R}{\rho}= 2^{k_0} $ (see the statement of Lemma \ref{lem-C4}), we have \\$R^{2}\frac{|B_R|_{\mu_\alpha}}{\rho^2|B_R\setminus B_\rho|_{\mu_\alpha}}\leq C_{k_0}$, which combined with Lemma
  \ref{lem-C3} implies
  \begin{equation}   \label{est-propo-2} \int_{A_{\rho,R}}G_{\gamma,\alpha}(x,y)|\vp(x)-\vp(y)|^2|u(y)-\hat u_\rho|^2\,dxdy\leq C_{k_0}\left( [u]^2_{R}-[u]_{\rho}^2\right).
   \end{equation}
  Combining  estimates \eqref{est-propo-1} and \eqref{est-propo-2}, and taking into account that $$|u(y)-\bar u_\rho|^2\leq 2(|u(y)-\hat u_\rho|^2+2|\bar u_\rho-\hat u_\rho|^2),$$ we deduce the desired bound for $I_4,$
       \begin{equation}
       \begin{split}
       I_4&\leq C_{k_0}\left( [u]^2_{R}-[u]_{\rho}^2\right)\left(1+R^{1+2\gamma}|B_R\setminus B_\rho|^{-1}_{\mu_\alpha}\,|B_\rho|^{-1}_{\mu_\alpha} \int_{A_{\rho,R}} G_{\gamma,\alpha}(x,y)|\vp(x)-\vp(y)|^2\,dxdy\right) \\&\leq C_{k_0}\left( [u]^2_{R}-[u]_{\rho}^2\right)\left(1+\frac{R^{2}}{\rho^2} \frac{ |B_R|_{\mu_\alpha}}{ |B_\rho|_{\mu_\alpha}}\right)\\&\leq C_{k_0}\left( [u]^2_{R}-[u]_{\rho}^2\right) % \\&\blu{ \leq C_{k_0}\left( [u]^2_{B_R}-[u]_{B_\rho}^2\right) ??}
       \end{split}\label{bdi4}\end{equation} %\textcolor{green}{question}by Lemma \ref{lem-C5}.
  where  $C_{k_0}$ is a constant that changes from line to line, and for the second inequality we have used
 \begin{equation*}   G_{\gamma,\alpha}(x,y)|\vp(x)-\vp(y)|^2\leq C G_{\gamma,\alpha}(x,y)\frac{|x-y|^2}{\rho^2}\leq C\frac{|x-y|^{1-2\gamma}}{\rho^2|x|^\alpha|y|^\alpha}\leq  C\frac{R^{1-2\gamma}}{\rho^2|x|^\alpha|y|^\alpha}.
 \end{equation*}

To conclude the proof we are left to estimate  $I_1$. Using $\psi$ as a test function in the weak formulation  of the PDE \eqref{PDE} (note that, as $\psi$ is compactly supported, it is equivalent to look at the weak formulation \eqref{weak-formulation} on $\mathbb R$) we obtain
  \begin{align*}
 % \int_{\R} \frac{f(x)}{|x|^\theta} \psi(x) \,dx=I_1+2\int_{\R\setminus B_R}\int_{B_R}G_{\gamma,\alpha}(x,y)[u(x)-u(y)]\psi(x)\,dxdy. \end{align*}  \begin{align*}\color{red}
  \int_{\R} \frac{f(x)}{|x|^\theta} \psi(x) \,dx=\frac12I_1+\int_{\R\setminus B_R}\int_{B_R}G_{\gamma,\alpha}(x,y)[u(x)-u(y)]\psi(x)\,dxdy. \quad \end{align*}
%
%\textcolor{green}{is there a factor of 2 missing?}
Notice that the left hand side above can be bounded by
  \begin{align}\label{estimate-ff} \int_{\R} \frac{|f(x)|}{|x|^\theta} |\psi(x)| \,dx\leq 2 \|u\|_{L^\infty(B_{2\rho})} \|f\|_{L^\infty(B_{2\rho})}\int_{B_{2\rho}}\,\frac{dx}{|x|^\theta}\leq C  \|u\|_{L^\infty(B_{2\rho})} \|f\|_{L^\infty(B_{2\rho})}\rho^{1-\theta}.  \end{align}

  In summary,  going back to  $\frac{R}{\rho}= 2^{k_0} $,  \eqref{eq:hole}, \eqref{bdi23} and \eqref{bdi4} we have proved that
  \begin{equation*}
  \begin{split}
[u]_{\rho}^2=   \int_{B_\rho} \int_{B_\rho}G_{\gamma,\alpha}(x,y)| u(x)-u(y)|^2\,dxdy&\leq C(k_0)\left( [u]^2_{R}-[u]_{\rho}^2\right)+
    C_0  \|u\|_{L^\infty(B_{2\rho})} \|f\|_{L^\infty(B_{2\rho})}\rho^{1-\theta}\\
   &\quad +C_0\int_{\R\setminus B_R}\int_{B_R}G_{\gamma,\alpha}(x,y)|u(x)-u(y)||\psi(x)|\,dxdy.
  \end{split}
  \end{equation*}

Now we use  Lemmas \ref{lem-C4} and \ref{lem-C6} to bound the last integral above, and this completes the proof of  Proposition \ref{Prop-C2}.

  %Hence, \begin{align} |I_1|&\leq C\rho^{1-\theta}\|f\|_{L^\infty(B_M)}+C\|u\|_{L^\infty} \int_{\R\setminus B_R(x_0)}\int_{B_\rho(x_0)}K(x,y)(\|u\|_{L^\infty}+|u(y)|)dxdy\\&\leq C_1\rho^{1-\theta}+.....,\end{align} where $C_1$ depends on the $L^\infty$ norm of $f$, $u$ on $B_M$, and the tail part of $u$, and also we need $\alpha$ to be  small.

\end{proof}

We conclude this subsection with one last technical lemma that will be used in the proof of Theorem \ref{regularity}.

\begin{lemma}\label{lem-38} We have
\begin{align*} [u]_{2}^2\leq   C\|u\|_{L^\infty(B_3)} \left( \|u\|_{L^\infty(B_3)} + \|f\|_{L^\infty(B_3)}+\Tail(u,x_0,3)\right).
\end{align*}
\end{lemma}\label{lem-C7}

\begin{proof}
Let $\phi\in C_c^{\infty}(B_3)$ be such that $\phi\equiv 1$ in $B_\frac 52$ and $|\nabla \phi|\leq 4$. We write
 \begin{equation*}\begin{split}
 [u(x)-u(y)]&[\phi^2(x)u(x)-\phi^2(y)u(y)]\\&
 =\phi^2(x)[u(x)-u(y)]^2+ [u(x)-u(y)][\phi(x)+\phi(y)] [\phi(x)-\phi(y)]u(y)\\&\geq |u(x)-u(y)|^2\left\{\phi^2(x)-\frac 1 8(\phi(x)+\phi(y))^2\right\}-2|\phi(x)-\phi(y)|^2u^2(y).
 \end{split}\end{equation*}
 Interchanging $x$ with  $y$ in the above inequality, and then summing with it we get
 \begin{equation}\begin{split} &2[u(x)-u(y)][\phi^2(x)u(x)-\phi^2(y)u(y)]\\&\geq |u(x)-u(y)|^2[\phi^2(x)+\phi^2(y)-\frac14(\phi(x)+\phi(y))^2]-2|\phi(x)-\phi(y)|^2[u^2(y)+u^2(x)]\\ & \geq \frac12|u(x)-u(y)|^2[\phi^2(x)+\phi^2(y)]-2|\phi(x)-\phi(y)|^2[u^2(y)+u^2(x)] \label{C7-1}.
 \end{split}\end{equation}
 Taking $\phi^2 u$ as a test function in the weak formulation of \eqref{PDE} (again, it is enough to look at the weak formulation  \eqref{weak-formulation}  on $\mathbb R$),
  \begin{equation}\label{equation100}\begin{split}
2    \int \frac{f(x)}{|x|^{\theta}}&u(x)\phi^2(x)\,dx\\
    &=\left(\int_{B_3}\int_{B_3} +\int_{\R\setminus B_3}\int_{B_3} +\int_{B_3}\int_{\R\setminus B_3}    \right)G_{\gamma,\alpha}(x,y)[u(x)-u(y)][\phi^2(x)u(x)-\phi^2(y)u(y)]\,dxdy\\
   &=:I_1+I_2+I_3.
   \end{split}\end{equation}
   It follows from \eqref{C7-1} that
   $$I_1\geq \frac12[u]_{2}^2-C  \|u\|^2_{L^\infty(B_3)},$$
    and, in addition,
    \begin{align}
    |I_2|+|I_3|\leq  C\|u\|_{L^\infty(B_3)} \left( \|u\|_{L^\infty(B_3)} +\Tail(u,x_0,3)\|\right),
    \end{align}
   thanks an argument similar to the one to obtain Lemma \ref{lem-C6}. For the left hand side of \eqref{equation100} we argue as in \eqref{estimate-ff}, which yields the desired conclusion.
\end{proof}

%============================================================

\subsection{Proof of Theorem \ref{regularity}} \label{section:pfregthem}

%\subsection{Regularity}

 With the same notations as in Proposition  \ref{Prop-C2}  we have
  \begin{align}\label{Thm-3.4.1} [u]_\rho^2&\leq\tau [u]^2_{2^{k_0}\rho}+\frac{C_0}{1+C(k_0)}\sum_{k=k_0+1}^{N+1}2^{-k\gamma}[u]_{2^k\rho}^2 +C_1\rho^{\sigma_1}, \end{align}
    for some $\sigma_1\in(0,1)$ (which can be made independent of $\theta$ and $\alpha$), the constant $C_1\geq1$ depends on $\|u\|_{L^\infty(B_M)}$, $\Tail(u,x_0,1)$ and $\|f\|_{L^\infty(B_M)}$.

    We claim that \eqref{Thm-3.4.1} implies that there exist $C_2\geq 1$  and  $\sigma_2\in (0,\sigma_1)$  such that
    \begin{align}\label{claim-C2} [u]_\rho^2\leq C_2\rho^{\sigma_2} \left([u]_2^2+1\right)\end{align}  for every  $0<\rho\leq2$.  In order to prove our claim,  for given constants $C_2\geq 1$ and $\sigma_2\in (0,\sigma_1)$ we set
       $$\mathcal B=\mathcal B_{C_2,\sigma_2}:=\big\{\tilde\rho\in(0,2]:\text{ \eqref{claim-C2} holds true for every }\rho\in[\tilde\rho,2]\big\}.$$
    Note that for $C_2\geq 8$ and for any $\sigma_2\in [0,1]$ we have  $[\frac18,2]\subset\mathcal B$. Our claim would follow if we show that for some suitable $C_2$ and $\sigma_2$ we have that the infimum
    $$\rho_0:=\inf {\mathcal B}$$
     is equal to $0$.
%
%    Let us first fix   $k_0\geq3$ (independent of $\gamma\geq\frac34$) satisfying
%       $$ \sum_{k=k_0+1}^\infty 2^{-k(\gamma- \frac14 )}\leq\frac{1}{4 C_0},$$
%       and then fix $\sigma_2\in \left(0,\min\{\sigma_1,\frac14\}\right)$ such that $$C_{k_0}\left(2^{k_0\sigma_2}-1\right)\leq\frac14.$$
%      We now take $C_2=8C_1(1+C(k_0))$ and
We    proceed by contradiction assuming that for this choice of parameters
it holds  $\rho_0>0$ .
Then we can evaluate  \eqref{Thm-3.4.1} at $\rho=\delta \rho_0 $ with  $\delta\in (2^{-k_0},1]$.
    Observing that    if $2^{N+1}\rho\leq 2$ then
        $[u]_{2^k\rho}^2\leq [u]_{2}^2$ for $k\leq N+1$, we have
    % then for $\delta\in (2^{-k_0},1]$ using \eqref{Thm-3.4.1} we estimate
    \begin{align} [u]_{\delta\rho_0}^2& \leq C_2 (\delta\rho_0)^{\sigma_2} \left( \tau  2^{k_0\sigma_2}+\frac{C_0}{1+C(k_0)}\sum_{k=k_0+1}^{N+1}2^{-k(\gamma-\sigma_2)} +\frac{C_1}{C_2}(\delta\rho_0)^{\sigma_1-\sigma_2}\right)\left([u]_2^2+1\right).  \end{align}
      We first fix  $k_0\geq3$ (independent of $\gamma\geq\frac34$ and $\sigma_2\leq\frac14$) satisfying
       $$ \sum_{k=k_0+1}^\infty 2^{-k(\gamma-\sigma_2)}\leq\frac{1}{4 C_0},$$
       and then fix $\sigma_2\in \left(0,\min\{\sigma_1,\frac14\}\right)$ such that $$C_{k_0}\left(2^{k_0\sigma_2}-1\right)\leq\frac14.$$
        Finally, taking $C_2=8C_1(1+C(k_0))$, and recalling that $\tau=\frac{C(k_0)}{1+C(k_0)}$ we obtain
         $$\tau  2^{k_0\sigma_2}+\frac{C_0}{1+C(k_0)}\sum_{k=k_0+1}^N2^{-k(\gamma-\sigma_2)} +\frac{C_1}{C_2}(\delta\rho_0)^{\sigma_1-\sigma_2}\leq1.$$
         Therefore, for these choice of parameters we  have that $(2^{-k_0}\rho_0,\rho_0]\subset\mathcal B$, which contradicts that $\rho_0$ is the infimum.
        Hence necessarily $\inf \mathcal B=0$,  concluding the proof of \eqref{claim-C2}.

Finally we prove our theorem by observing that for $|x_0|<\frac M4$ and $0<\rho\leq 1$ we have
\begin{align}
\int_{B_\rho}\left|u(x)   -\fint _{B_\rho}u(y)\,dy\right|^2\,dx
 &\leq \frac 1 {2\rho}\int_{B_\rho} \int_{B_\rho} |u(x)-u(y)|^2 \,dxdy \\&\leq C(M) \rho^{2\gamma}\int_{B_\rho} \int_{B_\rho}G_{\gamma,\alpha} |u(x)-u(y)|^2 \,dxdy \\&\leq C(M,C_2) \rho^{2\gamma+\sigma_2} ([u]_2^2+1),
 \end{align}
 where the last inequality is equivalent to estimate \eqref{claim-C2}.
The theorem follows immediately  from Lemma \ref{lem-38} and  the Campanato embedding (see e.g. \cite[Theorem 5.5]{GM}) with $\sigma=\frac14\sigma_2$ and $\gamma_0=\frac12(1-\frac12\sigma_2) $.

\section{Compactness results}\label{section:compactness}

\begin{comment}

{\color{red}

Observe that. if $\tilde u$ is a minimizer of  $\tilde E_{\gamma,\alpha,\beta}$, then $$\tilde u_\sigma(\cdot)=\sigma^{-\frac{1-2\gamma}{2}}\tilde u\left(\frac{\cdot}{\sigma}\right)$$ is also a minimizer.

Check where this goes}
\end{comment}

We  now work in the context of Theorem \ref{main-theorem}, namely we fix
$n=1$ and for $\ve>0$ small and we make  suitable choices of  $\gamma=\gamma_\ve\uparrow\frac12$, $\alpha=\alpha_\ve\to0$, $\beta=\beta_\ve\to0$ and $p=p_\ve\to\infty$. More precisely, we assume
\begin{align}\label{relations}   \ve p_\ve\to a\in(0,\infty),\quad \beta_\ve p_\ve\to b\in [0,1),\quad  \text{as} \quad\ve\downarrow 0.   \end{align}

 We denote  $\LL_\ve:=\LL_{\gamma_\ve,\alpha_\ve}$, $S_\ve:=S(\alpha_\ve,\beta_\ve)$ , $E_{\ve}:=E_{\gamma_\ve,\alpha_\ve,\beta_\ve}$, $G_\ve:=G_{\gamma_\ve,\alpha_\ve}$, where these quantities are given by \eqref{formula-L},  \eqref{eq-S}  \eqref{defi-E}, \eqref{defiG}, respectively.\\

In the context of Theorem \ref{main-theorem} we may assume that there exists an even  positive minimizer $u_\ve$ of \eqref{eq-S} which, in particular, is a solution of equation  \eqref{eq-extremal} (for the existence and properties of such minimizers recall Remark \ref{rem:existenceofparameters} and Theorem \ref{theorem:properties}).
\begin{comment}{main-theorem}
{rem:existenceofparameters}
In our hypothesis {\color{red}rephrase it by saying that by our choice of parameters the minimizers exits? maybe also  refer to some remark after statement of the main theorem. next line,   rephrase also ``assume that $u_\ve$ is radially symmetric''? it is not assumption, but property of minimizers for these parameters. ``radial symmetry''  $-->$ ''even''   } there exists a positive minimizer $u_\ve$ of \eqref{eq-S} which, in particular, is a solution of equation  \eqref{eq-extremal}.
We assume that $u_\ve$ is radially symmetric (even in dimension $n=1$), which holds, for instance, if $\alpha_\ve>0$.
\end{comment}
In this case, we also know (from  Theorem \ref{theorem:properties}) that $\tilde u_\ve(x):=|x|^{-\alpha_\ve} u_\ve(x)$ is radially decreasing
 and behaves like $\tilde u_\ve(x)\sim |x|^{-\alpha_\ve}$ as $|x|\to 0$.   The last statement follows from a minor modification of Proposition 4.5 in \cite{Dipierro-Montoro-Peral-Sciunzi}, where the authors show that $u_\ve\in L^\infty(\mathbb{R})$  in the case  $\alpha_\ve=\beta_\ve$.
 %The uniform bound in \cite{Dipierro-Montoro-Peral-Sciunzi}, depends on
% $\int_{\mathbb{R}}\frac{|u(x)|^{2^*_\alpha}}{|x|^{2^*_\alpha}}dx $, where $2^*_\alpha=\frac{2}{1-2\gamma}$ (the integral of the right-hand side of their equation times their solution).
 % uniformly {\color{red}uniformly in what sense?}{\color{blue} in the standard sense, it has uniform global bound. They prove decay that is faster than the singular factor, but may be good to double check this}  bounded in the case  $\alpha_\ve=\beta_\ve$.
 %However, their proof follows verbatim if we consider inequality \eqref{ineq_u} for general $\beta_\epsilon$ in our range (instead of the inequality obtained when $\beta_\ve= \alpha_\ve$ as they do) and 
% the bound depend on $\int_{\mathbb{R}} \frac{|u_\ve(x)|^{p}}{|x|^{\beta p}}\,dx$. (the right hand side of \eqref{eq-extremal} times $u_\ve$).} {\color{red}I think this is NOT true as the problem has some scale invariance.  Integral of the RHS of the PDE times $u_\ve$ goes to $0$ in our case. maybe we should give a proof or cite the result exactly the way stated in that paper }
%Note that  if  $ u$ is a minimizer of  $E_{\gamma,\alpha,\beta}$, then $ u_\sigma(\cdot)=\sigma^{-\frac{1-2\gamma}{2}}\tilde u\left(\frac{\cdot}{\sigma}\right)$ is also a minimizer. Hence, 
By
 replacing $u_\ve$ by $x\mapsto c_\ve u_\ve(R_\ve x)$ for some suitable $c_\ve>0$ and $R_\ve>0$ we can also assume that $u_\ve$ satisfies
\begin{align}\label{eq-uepsilon}   \LL_\ve u_\ve=\frac{1}{p_\ve}  \frac{1}{|x|^{\beta_\ve p_\ve}} (u_\ve)^{p_\ve-1},\quad   \max_{\overline {B_1}}u_\ve =1,\end{align}
as in the statement of Theorem \ref{main-theorem}.

The main objective of this section is to show that the function
\begin{equation}\label{eta-epsilon}\eta_\ve:=p_\ve(u_\ve-1)
\end{equation}
converges as $\ve\downarrow 0$.
%We need the following proposition to show that $\int_\R \frac{u_\ve^{p_\ve}}{|x|^{\beta_\ve p_\ve}}dx\leq C.$
Note that $\eta_\ve$ is a solution to
\begin{align}  \LL_\ve \eta_\ve=\frac{\left(1+\frac{\eta_\ve}{p_\ve}\right)^{p_\ve-1}}{|x|^{\beta_\ve p_\ve}}.
\end{align}

Our main result is, thus,

\begin{theorem}\label{thm-limit-function}  Up to a subsequence,     $\eta_\ve\to\eta_0$ in $C^0_{loc}(\R)$. Moreover, if we define
$$\eta:=\eta_0+\log\zeta_\frac12,$$
then $\eta\in L^1_\frac 12(\R)$ and it is a solution to
\begin{equation}\label{statement-thm}
(-\D)^\frac12\eta=|x|^{-b} e^\eta\quad\text{in }\R,\quad \kappa:= \int_\R |x|^{-b} e^\eta \,dx=2 \pi(1-b)<\infty.
\end{equation}
%Moreover,  $\eta$ satisfies  $$\eta\Big(\frac\lambda x\Big)=\eta(\lambda x)+2(1-b)\log |x|\quad\text{for } x\neq 0, $$ for some $\lambda>0.$ \textcolor{green}{is this straightforward to see?} {\color{blue}The proof is not straight forward..... maybe we can remove it.....   }
\end{theorem}

Note that the convergence statement of the previous result follows from Theorem \ref{regularity} as long as we can prove its hypotheses, namely a local uniform bound for $\eta_\varepsilon$ and a bound (independent of $\varepsilon$) for $\|\eta_\varepsilon \|_{L^1_q(\mathbb R)}$ with  $q=1+2\gamma_\varepsilon+\alpha_\varepsilon.$  We divide the proof of this last statement into several intermediate steps.\\

\begin{proposition} \label{propo-1}
We have, as $\ve \downarrow 0$,
\begin{equation}\label{S-bounded}S_\ve\leq \frac {C}{p_\ve}
\end{equation} and, in particular,
\begin{equation}
\label{estimate-u}
\int_\R \frac{(u_\ve)^{p_\ve}}{|x|^{\beta_\ve p_\ve}}\, dx\leq C.
\end{equation}\end{proposition}

\begin{proof}  For ${\delta_\ve}>0$ small (to be chosen later) we set
$$\vp_\ve(x)=(1+x^2)^{-{\delta_\ve}},$$
which will be used as a test function in \eqref{defi-E}.
Since $\vp$ is even
$$\mathcal E_\ve(\vp_\ve):=\int_{\R}\int_\R\frac{(\vp_\ve(x)-\vp_\ve(y))^2}{|x-y|^{1+2{\gamma_\ve}}|x|^{\alpha_\ve}|y|^{\alpha_\ve}}\,dydx
\leq4\int_0^\infty\int_0^\infty \frac{(\vp_\ve(x)-\vp_\ve(y))^2}{|x-y|^{1+2{\gamma_\ve}}|x|^{\alpha_\ve}|y|^{\alpha_\ve}}\,dydx. $$
For $x\geq0$ we set
$$f_\ve(x):=\int_0^\infty \frac{(\vp_\ve(x)-\vp_\ve(y))^2}{|x-y|^{1+2{\gamma_\ve}}|y|^{\alpha_\ve}}\,dy.$$
We claim that
\begin{align}  \label{est-6.4} f_\ve(x)\leq \frac{C{(\delta_\ve)}^2}{1+x^{2{\gamma_\ve}+{\alpha_\ve}+4{\delta_\ve}}},\quad x\geq0.\end{align}
First we prove the claim for $x\geq 4$. We write \begin{align*}   f_\ve(x)&= \left(\int_0^\frac x2 +\int_\frac x2^{2x}+\int_{2x}^\infty\right)   \frac{ (\vp_\ve(x)-\vp_\ve(y))^2}{|x-y|^{1+2{\gamma_\ve}}|y|^{\alpha_\ve}}\,dy\\&=:I_1+I_2+I_3.\end{align*}
We use that  for $y\geq x\geq4$,
\begin{align*}  |\vp_\ve(x)-\vp_\ve(y)|&=\left|\int_ x^y  \vp_\ve'(t)\,dt\right|\leq C{\delta_\ve}\int_x ^y t^{-1-2{\delta_\ve}}\,dt \\&\leq C\left( \frac{1}{x^{2{\delta_\ve}}}-\frac{1}{ y ^{2{\delta_\ve}}}\right) \\&=:C g_\ve(x,y).\end{align*}
  Changing the variable  as $y\mapsto x y$ yields
  \begin{align*}  I_3 &\leq C\int_{2x}^\infty \frac{1}{y^{1+2{\gamma_\ve}+{\alpha_\ve}}} g_\ve^2(x,y)\, dy\\&=\frac{ C}{x^{2{\gamma_\ve}+{\alpha_\ve}+4{\delta_\ve}}}\int_{2}^\infty \frac{1}{y^{1+2{\gamma_\ve}+{\alpha_\ve}}}g_\ve^2(1,y)\,dy.
  \end{align*}
    Now, integration by parts gives, taking into account that the constant $C$ stays uniformaly bounded in $\epsilon$ when we take the limit ${\gamma_\ve}\to\frac12$, ${\alpha_\ve}\to 0$ and ${\delta_\ve}\to 0$,  %{\color{red}below, there will be a boundary term on the RHS, irght?} {\color{blue} check if it is correct now}
  \begin{align*}
  \int_{2}^\infty \frac{1}{y^{1+2{\gamma_\ve}+{\alpha_\ve}}}g_\ve^2(1,y)\,dy&
  \leq C \int_{2}^\infty \frac{1}{y^{2{\gamma_\ve}+{\alpha_\ve}}}g_\ve(1,y)\partial_y g_\ve(1,y)\,dy
  {+ C\left(1-\frac{1}{2^{2\delta_\ve}}\right)^2} \\
  & \leq C {\delta_\ve} \int_{2}^\infty\frac{dy}{y^{2{\gamma_\ve}+{\alpha_\ve}+1+2{\delta_\ve}}}g_\ve(1,y)\,dy+ { C\delta_\ve^2}\\
  & \leq C\,{\delta_\ve} \int_{2}^\infty \frac{1}{y^{2{\gamma_\ve}+{\alpha_\ve}+2{\delta_\ve}}} \partial_y g_\ve(1,y)\,dy  {+ C\delta_\ve^2}\\\ &
  \leq  C\, ({\delta_\ve})^2\end{align*}
  and hence,
  $$I_3\leq \frac{C({\delta_\ve})^2}{x^{2{\gamma_\ve}+{\alpha_\ve}+4{\delta_\ve}}}.$$
    In a similar way, changing the variable $y\mapsto \frac xy$, we obtain
     \begin{align*}I_1&\leq  \frac{C}{x^{1+2{\gamma_\ve}}} \int_0^\frac x2 y^{-{\alpha_\ve}}(\vp_\ve(x)-\vp_\ve(y))^2\,dy\\&= \frac{C}{x^{2{\gamma_\ve}+{\alpha_\ve}+4{\delta_\ve}}} \int_2^\infty y^{{\alpha_\ve}-2}  g_\ve^2(1,\tfrac 1y)\,dy  \\ &\leq \frac{C({\delta_\ve})^2}{x^{2{\gamma_\ve}+{\alpha_\ve}+4{\delta_\ve}}}.
     \end{align*}
     Finally, using that
     $$|\vp_\ve(x)-\vp_\ve(y)|\leq C{\delta_\ve}\frac{|x-y|}{x^{1+2{\delta_\ve}}},\quad \text{for}\quad \frac x2\leq y\leq 2x,$$
     we bound
     \begin{align*}
     I_2\leq \frac{C{(\delta_\ve)}^2}{x^{2+4{\delta_\ve}+{\alpha_\ve}}}\int_\frac x2^{2x} |x-y|^{1-2{\gamma_\ve}}\,dy
     \leq \frac{C({\delta_\ve})^2}{x^{2{\gamma_\ve}+{\alpha_\ve}+4{\delta_\ve}}}.
      \end{align*}
      Combining the above estimates  yields the inequality in \eqref{est-6.4} for $x\geq4$.  Let us now  consider  the case $0\leq x\leq4$. We write
      \begin{align*}   f_\ve(x)&= \left(\int_0^8 +  \int_{8}^\infty\right)   \frac{ (\vp_\ve(x)-\vp_\ve(y))^2}{|x-y|^{1+2{\gamma_\ve}}|y|^{\alpha_\ve}}\,dy\\
      &=:J_1+J_2.
      \end{align*}
    As in the estimate of $I_2$, taking into account that $|\vp_\ve(x)-\vp_\ve(y)|\leq C\delta_\ve|x-y|$, one  obtains $J_1\leq C({\delta_\ve})^2. $
    To bound $J_2$ note first that for $0\leq x\leq y$ holds
    % we observe that  for $0\leq x\leq 4$ and $y\geq 8$
    \begin{align*}
   | \vp_\ve(x)-\vp_\ve(y)|=\left|2{\delta_\ve}\int_x^y \frac{t}{(1+t^2)^{1+{\delta_\ve}}}\, dt \right|.%\leq C{\delta_\ve} y^\frac14\int_0^\infty\frac{dt}{(1+t)^{\frac 54+2{\delta_\ve}}}\leq C{\delta_\ve} y^\frac14,
    \end{align*}
    Since for
  $0\leq x\leq t\leq y$ holds
  $$\frac{t}{(1+t^2)^{1+\delta_\ve}}\leq  C\frac{t^\frac14 t^\frac 34}{(1+t)^{2+2\delta_\ve}}\leq Cy^\frac14 \frac{(1+t)^\frac34}{(1+t)^{2+2\delta_\ve}} =Cy^\frac14 \frac{ 1}{(1+t)^{\frac 54+2\delta_\ve}},$$
  we have
       \begin{align*}
   | \vp_\ve(x)-\vp_\ve(y)|%=\left|2{\delta_\ve}\int_x^y \frac{t}{(1+t^2)^{1+{\delta_\ve}}}\, dt \right|
   \leq C{\delta_\ve} y^\frac14\int_0^\infty\frac{dt}{(1+t)^{\frac 54+2{\delta_\ve}}}\leq C{\delta_\ve} y^\frac14.
    \end{align*}

    Then, as $\frac12+2{\gamma_\ve}+{\alpha_\ve}\to \frac32$,
    $$J_2\leq C \int_8^\infty  \frac{ (\vp_\ve(x)-\vp_\ve(y))^2}{y^{1+2{\gamma_\ve}+{\alpha_\ve}} }\,dy\leq  C{(\delta_\ve)}^2  \int_8^\infty  \frac{  1}{y^{\frac12+2{\gamma_\ve}+{\alpha_\ve}} }\,dy\leq C{(\delta_\ve)}^2.$$
    This proves  inequality   \eqref{est-6.4}.\\

Now we chose $\delta_\ve>0$ such that $2\gamma_\ve+2\alpha_\ve+4\delta_\ve =1+\frac{4}{p_\ve}$.  Then, recalling the definition for $p_\ve$ in \eqref{pdef}, one obtains
$$2\delta_\ve=\frac{3-p_\ve\beta_\ve}{p_\ve}\in\left( \frac{2}{p_\ve},\frac{4-b}{p_\ve}\right),$$
thanks to our choice \eqref{relations}.  Therefore, on one hand, using \eqref{est-6.4} we get that
\begin{align*} \mathcal E_\ve(\vp_\ve) \leq C(\delta_\ve)^2 \int_0^\infty \frac{1}{1+x^{1+\frac{4}{p_\ve}}}\,dx \leq C (\delta_\ve)^2 \, p_\ve\leq\frac{C}{p_\ve}.\end{align*}
On the other hand,
\begin{align*}
 \int_0^\infty \frac{(\vp_\ve)^{p_\ve}}{x^{\beta_\ve p_\ve}}\,dx  \approx \int_0^\infty \frac{1}{x^{\beta_\ve p_\ve}} \frac{1}{(1+x)^{3-\beta_\ve p_\ve}}\,dx,
 \end{align*}
 and this integral is bounded above and below by a positive constant independent of $\ve$.
  Combining the above estimates yields \eqref{S-bounded}. The second part of the proposition follows from
   $$\frac{C}{p_\ve}\geq S_\ve=\frac{\displaystyle\int_\R u_\ve\LL_\ve u_\ve\, dx}{\displaystyle\left( \int_\R  \frac{(u_\ve)^{p_\ve}}{|x|^{\beta_\ve p_\ve}} \, dx\right)^\frac{2}{p_\ve}}=\frac{1}{p_\ve}\left( \int_\R  \frac{(u_\ve)^{p_\ve}}{|x|^{\beta_\ve p_\ve}}\, dx\right)^{1-\frac{2}{p_\ve}},$$
  and this completes the proof.\\
\end{proof}

Now we give estimates for the function $\eta_\ve$ defined in \eqref{eta-epsilon}.

%\textcolor{green}{I have added the lemma below}
\begin{lemma}  We have  that
\begin{align}\label{bound1}  \LL_\ve \eta_\ve=\frac{\left(1+\frac{\eta_\ve}{p_\ve}\right)^{p_\ve-1}}{|x|^{\beta_\ve p_\ve}} =\frac{O(1)}{|x|^{\beta_\ve p_\ve}},
\end{align}
and
\begin{align}\label{bound2} \eta_\ve\leq C(r)\,\text{ in } B_r,\qquad \max_{\overline{B_1}}\eta_\ve=0,
\end{align} for every $r>0$.
\end{lemma}

\begin{proof}
 Recall that  $u_\ve\leq1$ on $B_1$
 and,  since $x\mapsto |x|^{-\alpha_\ve}u_\ve(x)$ is monotone decreasing in $|x|$, we also have
\begin{align}\label{outside}  u_\ve(x)\leq u_\ve(1) |x|^{\alpha_\ve}\leq |x|^{\alpha_\ve}, \quad \text{for }|x|\geq 1.
\end{align}
As a consequence one obtains \eqref{bound1} since, for $1\leq |x|\leq r$, we have
\begin{equation}\label{bound4}(u_\ve)^{p_\ve-1}\leq |x|^{\alpha_\ve (p_\ve-1)}\leq r^{\beta_\ve p_\ve -\alpha_\ve}\leq r^{b+1},
\end{equation}
  as $\beta_\ve p_\ve\to b $. As for the bound \eqref{bound2}, just  note that $\eta_\ve\leq0$ on $B_1$, while   for $|x|\geq1$,
\begin{equation}\label{bound3}\eta_\ve(x)\leq p_\ve(|x|^{\alpha_\ve}-1) \leq\frac{C}{\alpha_\ve}(|x|^{\alpha_\ve}-1).
\end{equation}
\end{proof}

\begin{lemma} \label{lem-outside}  For any $q>1$ we have
\begin{equation}\label{eta-bound1}\int_\R\frac{\eta_\ve^+(x)}{1+|x|^q}\,dx\leq C(q).
\end{equation}
\end{lemma}
\begin{proof} We first recall the bound \eqref{bound3}  and our previous discussion.
 Notice that  $$\frac{1}{\alpha_\ve}(|x|^{\alpha_\ve}-1)\leq \frac{1}{\alpha_0}(|x|^{\alpha_0}-1)\quad\text{for }|x|\geq1,\, 0<\alpha_\ve\leq\alpha_0.$$
 Therefore, choosing  $\alpha_0 =\frac12(q-1)$   we get that, for $\alpha_\ve<\frac12(q-1)$,
 $$\int_{\R} \frac{\eta_\ve^+(x)}{1+|x|^{q}}\,dx\leq C(q)\int_{\R} \frac{ 1}{1+|x|^{\frac{1+q}{2}}}\,dx\leq C(q), $$
 as desired.
 \end{proof}

In order to improve the previous bound to obtain the necessary result to apply Theorem \ref{regularity} we need
the following technical lemma, which gives estimates for suitable  test functions, when $\gamma_\ve\geq\frac14$.
We will obtain the desired  (uniform in $\varepsilon$) bound for $\|\eta_\varepsilon \|_{L^1_q(\mathbb R)}$  in Lemma  \ref{lem-L1/2}.

\begin{lemma}\label{lem-est-1}Let $\vp\in C_c^\infty(B_2)$ be such that $\vp\geq0$ and $\vp\equiv1$ on $B_1$. Then $\LL_\ve\vp(x)<0$ for $|x|\geq2$ and, for some $C_1>0$,
 $$ \frac{1}{C_1} \frac{\chi_{B_3^c}(x)}{|x|^{1+2\gamma_\ve+\alpha_\ve}}\leq|\LL_\ve \vp(x)|\leq C_1\left(\frac{\chi_{B_1}(x)}{|x|^{\alpha_\ve}}+\frac{\chi_{B_1^c}(x)}{|x|^{1+2\gamma_\ve+\alpha_\ve}}\right).$$
\end{lemma}

\begin{proof}
For $ x\in (B_\frac12\setminus\{0\})\cup B_3^c$ we easily get from the definition that
$$|\LL_\ve \vp(x)|\leq C\left(\frac{\chi_{B_\frac12}(x)}{|x|^{\alpha_\ve}}+\frac{\chi_{B_3^c}(x)}{|x|^{1+2\gamma_\ve+\alpha_\ve}}\right),\quad\text{for }  \gamma_\ve\geq\frac14.$$
    To complete the proof we note that
\begin{align*}
|x|^{\alpha_\ve}\LL_\ve\vp(x)&=
\int_\R\frac{\vp(x)-\vp(y)}{|x-y|^{1+2\gamma_\ve} |y|^{\alpha_\ve}}\,dy=\int_\R\frac{\vp(x)-\vp(x-y)}{|y|^{1+2\gamma_\ve} |x-y|^{\alpha_\ve}}\,dy
=\int_\R\frac{\vp(x)-\vp(x+y)}{|y|^{1+2\gamma_\ve} |x+y|^{\alpha_\ve}}\,dy,\end{align*} which leads to $$2|x|^{\alpha_\ve}\LL_\ve \vp(x)=\int_\R\frac{1}{|y|^{1+2\gamma_\ve}} \left(\frac{\vp(x)-\vp(x-y)}{ |x-y|^{\alpha_\ve}}+\frac{\vp(x)-\vp(x+y)}{ |x+y|^{\alpha_\ve}}\right)\,dy.$$ As $\vp$ is smooth, we get that $$\frac{\vp(x)-\vp(x-y)}{ |x-y|^{\alpha_\ve}}+\frac{\vp(x)-\vp(x+y)}{ |x+y|^{\alpha_\ve}}=O(|y|^2)\quad\text{on }(x,y)\in (B_3\setminus B_\frac12)\times B_\frac14.$$
The result follows immediately. % by observing that for $x\in B_3$ the integrand is not 0 only for either $y>x-2$ ( and $ x>3$)  or $-y<2+x$ ( and  $x<-3$) and that for any fixed $x\in B_3\setminus B_\frac12$ the function $f_x(z)=|x+z|^{-\alpha_\ve}$ is (uniformly) smooth for $z\in  B_\frac14.$
\end{proof}

Now we  can improve the bound \eqref{eta-bound1}.

\begin{lemma}\label{lem-L1/2}
We have   $$ \int_{\R}\frac{|\eta_\ve|}{1+|x|^{1+2\gamma_\ve+\alpha_\ve}}\,dx\leq C,$$
for some constant $C$ independent of $\ve$.
\end{lemma}

\begin{proof} Let $\vp$ be a test function satisfying $\vp\in C_c^\infty(B_2)$, $\vp\geq0$ and $\vp\equiv1$ on $B_1$.  It holds
$$\int_\R \eta_\ve \LL_\ve \vp \,dx=\int_{\R}\vp\LL_\ve \eta_\ve \, dx=\int_\R \frac{\left(1+\frac{\eta_\ve}{p_\ve}\right)^{p_\ve-1}}{|x|^{\beta_\ve p_\ve}} \vp \,dx\leq C\int _\R \frac{\vp}{|x|^{\beta_\ve p_\ve}}\,dx\leq C,$$ where the first inequality follows using the bound \eqref{bound4}.
Hence,  by Lemmas \ref{lem-outside} and  \ref{lem-est-1} we obtain \begin{align}\label{est-Ls}\int_{B_3^c}\frac{|\eta_\ve|}{1+|x|^{1+2\gamma_\ve+\alpha_\ve}}\,dx\leq C+C\int_{B_3}\frac{|\eta_\ve|}{|x|^{\alpha_\ve}}\,dx.\end{align}

We now assume by contradiction that the lemma is false. Then necessarily
$$\int_{B_3}\frac{|\eta_\ve|}{|x|^{\alpha_\ve}}\,dx\to\infty\quad\text{as}\quad \ve\to 0.$$
Since
\begin{equation}\label{Holder}\int_{B_3}\frac{|\eta_\ve|}{|x|^{\alpha_\ve}}\,dx\leq C\|\eta_\ve\|_{L^2(B_3)},
\end{equation}
thanks to H\"older's inequality,
we also have, in particular, that
$$\mu_\ve:=\|\eta_\ve\|_{L^2(B_3)}\to\infty,$$
as $\ve\downarrow 0$.
Then, setting
$$\tilde\eta_\ve:=\frac{\eta_\ve}{\mu_\ve},$$
 one obtains that for every $r>0$ (recall \eqref{bound1} and \eqref{bound4})
$$\LL_\ve\tilde\eta_\ve=\frac{o_\ve(1)}{|x|^{\alpha_\ve}}\quad\text{in }B_r$$
and
$$\int_{\R}\frac{|\tilde \eta_\ve|}{1+|x|^{q_\ve}}\,dx\leq C,\quad q_\ve:=1+2\gamma_\ve+\alpha_\ve, $$
where this last inequality follows from \eqref{est-Ls} and \eqref{Holder}. Moreover, using the monotonicity of $u_\ve$ one gets that for any $M\geq 1$ \begin{align} \|\tilde\eta_\ve
\|_{L^\infty(B_M)}\leq C(M)\left(1+\int_{\R}\frac{|\tilde \eta_\ve|}{1+|x|^{q_\ve}}\,dx \right)\leq C(M).\end{align}
 Also note that $$\max_{\bar B_1}\tilde\eta_\ve=0,\quad \tilde\eta_\ve ^+=o_\ve(1)\quad\text{in }B_R,$$ for every $R>0$.
 Therefore, by \eqref{claim-1}, up to a subsequence,  $\tilde\eta_\ve\to\tilde\eta$ in $C^0_{loc}(\R)$, where $$\tilde\eta\leq 0,\quad\max_{\bar B_1}\tilde\eta=0,\quad  \|\tilde\eta\|_{L^2(B_3)}=1,\quad \int_\R\frac{|\tilde\eta|}{1+x^2}\,dx<\infty. $$
In order to obtain a contradiction we will show that
\begin{equation}\label{eq-tilde}(-\D)^\frac12\tilde\eta=0 \quad\text{in } \R.
\end{equation}
Since the only solution of this equation is  $\tilde\eta\equiv const$ (see e.g. proof of Lemma 2.4 in \cite{Hyd}), this yields a contradiction to the fact that $\max_{\bar B_1}\tilde \eta =0$ and  $\|\tilde\eta\|_{L^2(B_3)}=1$. Thus, it remains to obtain \eqref{eq-tilde}, the equation  satisfied by $\tilde \eta$.

To see this note that, up to a further  subsequence, the following limit exists
 \begin{equation*}\label{pc_0}
\lim_{R\to\infty}\lim_{\ve\to0}\int_{B_R^c}\frac{|\tilde\eta_\ve|}{|x|^{q_\ve}}\,dx=:c_0\geq 0.
\end{equation*}
Note that we can rewrite $c_0$ as
\begin{equation}\label{c_00}
c_0=-\lim_{R\to\infty}\lim_{\ve\to0}\int_{B_R^c}\frac{\tilde\eta_\ve}{|x|^{q_\ve}}\,dx. \end{equation}
This follows as $q_\ve\to 2$, and therefore
$$\int_{\{|x|\geq R\}} \frac{\eta_\ve^+}{|x|^{q_\ve}}\,dx\leq \frac{1}{R^\frac14}\int_{B_R^c}\frac{\eta_\ve^+}{|x|^\frac 54}\, dx\leq \frac{C}{R^\frac14}\xrightarrow{R\to\infty}0,$$
thanks to Lemma \ref{lem-outside}. Let us first show that $\tilde\eta $ satisfies the equation
$${(\varsigma_{1/2})^{-1}}(-\D)^\frac12\tilde\eta+c_0=0\quad\text{in }\R,$$
in the sense that for every $\vp\in C_c^\infty(\R)$,
 \begin{align}\label{eq-bareta}
{(\varsigma_{1/2})^{-1}}\int_{\R}\tilde\eta(-\D)^\frac12\vp \,dx+c_0\int_{\R}\vp \,dx=0.\end{align}
To this end we fix $R>>1$ such that $\vp$ is supported in $B_R$. Then for $|x|\geq r>>R$ we see that $$\LL_\ve\vp(x)=\int_{\mathbb R} \frac{-\vp(\tilde x)}{|x-\tilde x|^{1+2\gamma_\ve}|x|^{\alpha_\ve}|\tilde x|^{\alpha_\ve}}\,d\tilde x=-\frac{1+o_\ve(1)+o_r(1)}{|x|^{q_\ve}}\int_{\R}\vp \,d\tilde x.$$
We write
\begin{align*} o_\ve(1)=\frac{1}{\mu_\ve}\int_{\R}\frac{\left(1+\frac{\eta_\ve}{p_\ve}\right)^{p_\ve-1}}{|x|^{b_\ve p_\ve}}\vp \,dx&=\int_\R \tilde\eta_\ve \LL_\ve\vp \,dx\\ &=\int_{B_r} \tilde\eta_\ve \LL_\ve\vp \,dx -(1+o_r(1)+o_\ve(1))\int_\R\vp \,d\tilde x\int_{B_r^c}   \frac{\tilde\eta_\ve}{|x|^{q_\ve}}\,dx.
\end{align*}
Taking   $\ve\to0$ first, and then taking $r\to\infty$ we obtain \eqref{eq-bareta}. Recall that we have used the expression for $c_0$  given in  \eqref{c_00}.\\

It remains to show that  $c_0=0$. In order to prove this we fix $\vp\in C_c^\infty(B_2)$ such that $\vp\geq0$ and $\vp\equiv 1$ in $B_1$. For $t>0$ small, setting $\vp_t(x):=\vp(t x)$ we see that $\int_\R\vp_t \,dx\approx\frac1t$, and
$$ (-\D)^\frac12\vp_t (x)=t [ (-\D)^\frac12\vp] (t x)=O(1)\frac{t}{1+t^2x^2}\leq C \frac{t}{1+t^2x^2} ,$$
where we have used that for a smooth compactly supported $\varphi$ it holds
 $$|(-\Delta)^\frac12\vp(x)|\leq \frac{C}{1+|x|^2}$$
Hence, using it as a test function in \eqref{eq-bareta} we obtain
$$c_0\int_{\R}\vp_t\, dx=- {(\varsigma_{1/2})^{-1}}\int_\R\tilde\eta(-\D)^\frac12\vp_t \,dx\leq C \int_{\R}|\tilde \eta| |(-\D)^\frac12\vp_t |\,dx.$$

Now note that for $t>0$ small $$\frac{t^2|\tilde\eta(x)| }{1+t^2 x^2}\leq \frac{|\tilde\eta(x)|}{1+x^2}, \quad x\in\R.$$
Then,
if $c_0>0$, then  multiplying both sides  by $t$ and using the dominated convergence theorem as $t\to 0$, %and recalling that $|(-\D)^\frac12\vp(x)|\leq \frac {C}{1+x^2}$,
 we get
$$c_0\leq C  \int_{\R}\frac{|\tilde\eta(x)|t^2}{1+t^2x^2}\,dx%=\int_\R\frac{t^2|\tilde\eta(x)|\chi_{|x|\leq\frac2t}}{1+t^2 x^2}\,dx
\xrightarrow{t\to0}0,$$
which is a contradiction.
%The above last conclusion follows from the dominated convergence theorem, as $$\frac{t^2|\tilde\eta(x)|  }{1+t^2 x^2}\leq \frac{|\tilde\eta(x)|}{1+x^2}, \quad x\in\R,$$ for $t>0$ small.
This concludes the proof of Lemma \ref{lem-L1/2}.\\
\end{proof}

 The bound from Lemma \ref{lem-L1/2} allow us to apply Theorem \ref{regularity} to solutions to  Equation \eqref{bound1}, obtaining the following H\"older estimate.

 \begin{align}\label{claim-1}\|\eta_\ve\|_{C^{0,\sigma}(B_R)}\leq C(\sigma,R) \end{align}
 for some $\sigma>0$ (which depends on $b:=\lim_{\ve\to0} \beta_\ve p_\ve$). Then, as a consequence of Proposition \ref{propo-1}, Lemma \ref{lem-L1/2} and \eqref{claim-1} we get that, up to a subsequence,{ $$\eta_\ve\to\eta_0\quad\text{in }C^0_{loc}(\R),\quad  \quad \int_\R\frac{|\eta_0|}{1+x^2}\,dx<\infty.$$
Now define
 $$\eta:=\eta_0+\log\varsigma_\frac12\in L^1_\frac 12(\R)$$
in order to simplify the multiplicative constant.}
To conclude the proof of Theorem  \ref{thm-limit-function} we are left to show that the limit function $\eta$ satisfies the desired equation. However, proceeding as in the proof of Lemma \ref{lem-L1/2}
before,  one can show that  this limit function $\eta$ satisfies
 $$(-\D)^\frac12\eta+c_1=  \frac{e^\eta}{|x|^b}\quad\text{in }\R  ,$$ for some $c_1\geq 0$. Since the $e^\eta |\cdot|^{-b}\in L^1(\R)$, following again  the proof of the above lemma we would get that $c_1=0$, and this shows \begin{align}\label{eq-app-1}   (-\D)^\frac12 \eta=|x|^{-b} e^\eta\quad\text{ in }\R,\quad  \kappa:=\int_\R |x|^{-b} e^\eta\, dx<\infty,\end{align} for some $b<1$.
 %Then $\eta\in C^\infty(\R\setminus\{0\})\cap C_{loc}^\alpha(\R)$ for some $\alpha>0$, $\kappa=2\pi(1+\vartheta)$, and $$\lim_{|x|\to\infty}\frac{\eta(x)}{\log |x|}=-2(1+\vartheta).$$
Finally, note that all solutions in $L^1_{1/2}(\mathbb R)$ are of the form
are given by \eqref{solution-eta} (see \cite{Galvez-Jimenez-Mira,Zhang-Zhou}). Moreover, a direct calculation from the explicit expression shows that $\kappa=2\pi(1-b)$, and this finishes the proof of Theorem \ref{thm-limit-function}.
\qed

\section{Proof of Theorem \ref{main-theorem} }\label{section:proof}

We take $u_\ve$ the minimizer of \eqref{eq-S} and all the parameters as in Section \ref{section:compactness}. For $v\in C_c^\infty(\R)$ we set  $$w_\ve :=(1+\ve v)u_\ve ,\quad \lambda_\ve:=\frac12E_\ve[u_\ve]=\int_{\mathbb R} u_\ve \LL_\ve u_\ve \, dx=\frac{1}{p_\ve}\int_{\mathbb R} \frac{(u_\ve)^{p_\ve}}{|x|^{\beta_\ve p_\ve}}\,dx.  $$ %We set \begin{equation*} \lambda_\ve := E[u_\ve ] \end{equation*}

Using the integration by parts and the product formulas  \eqref{integration-by-parts} and \eqref{product-formula}, respectively, we compute
\begin{equation}\label{long-integral}
\begin{split}
\int_{\mathbb R} w_\ve\mathcal L_\ve w_\ve\,dx&=\int_{\mathbb R} u_\ve\mathcal L_\ve u_\ve\,dx+\ve \int_{\mathbb R} u_\ve\mathcal L_\ve(u_\ve v)\,dx+\ve \int_{\mathbb R} u_\ve v\mathcal L_\ve u_\ve\,dx+\ve ^2\int_{\mathbb R} u_\ve v\mathcal L_\ve (u_\ve v)\,dx\\
&=\int_{\mathbb R} u_\ve \mathcal L_\ve u_\ve  \,dx+2\ve \int_{\mathbb R} u_\ve v\mathcal L_\ve u_\ve\,dx+\ve  ^2\int_{\mathbb R} u_\ve v\mathcal L_\ve (u_\ve v)\,dx\\
&=:\lambda_\ve +2\ve  I_1+\ve ^2 I_2.
\end{split}
\end{equation}
For the $I_1$ integral, we recall  that  $u_\ve$ satisfies equation \eqref{eq-uepsilon}, so that
\begin{equation*}
I_1=\frac{1}{p_\ve}\int_{\mathbb R}  v\frac{(u_\ve)^{p_\ve  }}{|x|^{\beta_\ve p_\ve }}\,dx.
\end{equation*}
To estimate $I_2$ we use the product formula \eqref{product-formula}, thus
\begin{align*}  I_2&=\int_{\mathbb R} v^2 u_\ve  \LL_\ve  u_\ve \,dx+\int_{\mathbb R} (u_\ve)^2 v\LL_\ve v \,dx-\int_{\mathbb R}\int_{\mathbb R} u_\ve(x)v(x)[u_\ve(x)-u_\ve(y)][v(x)-v(y)]G_\ve(x,y)\,dydx\\ &=: I_{2,1}+I_{2,2}-I_{2,3}. \end{align*}

We have shown in Theorem \ref{thm-limit-function}  that for some  { $\eta_0\in C^0(\R)$ we have    \begin{equation}\label{limit}p_\ve(u_\ve-1)\to \eta_0\quad\text{ in }C^0_{loc}(\R),\end{equation}} as $\ve\downarrow0$. In particular, $u_\ve\to1$ in $C^0_{loc}(\R)$. Therefore, recalling that $u_\ve\leq |x|^{\alpha_\ve}$ on $B_1^c$, again by equation \eqref{eq-uepsilon} we get
 $$I_{2,1}=\frac{1}{p_\ve}\int_{\mathbb R} v^2\frac{(u_\ve)^{p_\ve}}{|x|^{\beta_\ve p_\ve}} \,dx=o_\ve(1)\quad \text{and }I_{2,3}=o_\ve(1).$$  Consequently, $$ I_2\xrightarrow{\ve\to0} \int_{\mathbb R} v\,\LL_{\frac 12,0} v \,dx.$$
Moreover, from \eqref{limit} we also have
$$ (u_\ve)^{p_\ve}=\big(1+\tfrac{1}{p_\ve} p_\ve (u_\ve-1)\big)^{p_\ve}\to { e^{\eta_0}}$$
and
$$|1+\ve v|^{p_\ve}=\big|1+\tfrac{1}{p_\ve} (\ve p_\ve v)\big|^{p_\ve}\to e^{av},$$
as $\ve \downarrow 0$. Therefore, by Fatou's lemma
 \begin{align}\label{conv-1} \int e^{av} \frac{{e^{\eta_0}}}{|x|^b}\,dx \leq \liminf_{\ve\to0}\int \frac{|w_\ve|^{p_\ve}}{|x|^{\beta_\ve p_\ve}}\,dx.
 \end{align}
In addition, we have shown in Proposition \ref{propo-1}  that
 $$\int_\R \frac{(u_\ve)^{p_\ve}}{|x|^{\beta_\ve p_\ve}}\,dx\leq C,$$
 then, up to a subsequence, we also have
 \begin{align} \label{conv-2}\int_{\mathbb R} \frac{(u_\ve)^{p_\ve}}{|x|^{\beta_\ve p_\ve}}\,dx =\bar \kappa+o_\ve(1),
 \end{align} for some constant $\bar\kappa\geq0$.
 Using again Fatou's lemma we can conclude
{ $$\bar\kappa\geq \kappa_0=:\int_\R \frac{e^{\eta_0}}{|x|^b}\,dx.$$} Indeed
  $$\kappa_0=\int_\R \frac{e^{\eta_0}}{|x|^b} \,dx\leq\liminf_{\ve\to0}\int_\R \frac{\left(1+\frac{\eta_\ve}{p_\ve}\right)^{p_\ve}}{|x|^{\beta_\ve p_\ve}}\, dx=\liminf_{\ve\to0}\int_{\mathbb R} \frac{(u_\ve)^{p_\ve}}{|x|^{\beta_\ve p_\ve}}\,dx =\bar \kappa. $$
  At the end of the proof we shall show that, actually,  {$\bar\kappa=\kappa_0$} and equality holds in \eqref{conv-1}.\\

Now, from equation \eqref{eq-uepsilon} we see that
 $$\lambda_\ve=\int_{\mathbb R} u_\ve\LL_\ve u_\ve \,dx=\frac{1}{p_\ve}\int_{\mathbb R} \frac{(u_\ve)^{p_\ve}}{|x|^{\beta_\ve p_\ve}}\,dx
 =\frac{1}{p_\ve}(\bar\kappa+o_\ve(1)). $$
Note also that
$$\int_\R v \frac{(u_\ve)^{p_\ve}}{|x|^{\beta_\ve p_\ve}}\,dx\to\int_\R {v\frac{e^{\eta_0}}{|x|^b}}\,dx$$ as $v$ is compactly supported.
 Combining these estimates we deduce from \eqref{long-integral} that
  $$\int_{\mathbb R} w_\ve \LL_\ve w_\ve \,dx=\lambda_\ve + \frac{2\ve}{p_\ve}(1+o_\ve(1))\int_{\mathbb R} v  {\frac{e^{\eta_0}}{|x|^b}}\,dx+\ve^2 (1+o_\ve(1))\int_{\mathbb R} v\LL_{\frac12,0}v \,dx.$$
Since  $u_\ve$ is a minimizer of \eqref{eq-S}, we get that
  \begin{align*}
\frac{  \displaystyle\int_{\mathbb R} w_\ve\LL_\ve w_\ve \,dx}
{  \displaystyle\left( \int_{\mathbb R} \frac{|w_\ve|^{p_\ve}}{|x|^{\beta_\ve p_\ve}} \,dx \right)^\frac{2}{p_\ve}}  \geq  \frac{  \displaystyle\int_{\mathbb R} u_\ve\LL_\ve u_\ve \,dx}{  \displaystyle\left( \int_{\mathbb R} \frac{(u_\ve)^{p_\ve}}{|x|^{\beta_\ve p_\ve}} \,dx \right)^\frac{2}{p_\ve}} =\frac{\lambda_\ve}{\left( \bar \kappa+o_\ve(1)\right)^{\frac{2}{p_\ve}}}, \end{align*}
 which together with $\lambda_\ve p_\ve\to\bar \kappa$ and  $\ve p_\ve \to a$ leads to
  \begin{align*}
   \frac{1+o_\ve(1)}{\bar \kappa}\int_{\mathbb R}  \frac{|w_\ve|^{p_\ve}}{|x|^{\beta_\ve p_\ve}}\,dx   &\leq \left(\frac{1}{\lambda_\ve}\int_{\mathbb R} w_\ve \LL_\ve w_\ve \,dx\right)^\frac{p_\ve}{2} \\ &=\left(  1+     \frac{2\ve}{\lambda_\ve p_\ve}(1+o_\ve(1))\int_{\mathbb R} v {\frac{e^{\eta_0}}{|x|^b}}\,dx+\frac{\ve^2}{\lambda_\ve} (1+o_\ve(1))\int_{\mathbb R} v\LL_{\frac12,0} v\, dx \right)^\frac{p_\ve}{2} \\ &=\left(  1+     \frac{2 a}{\bar\kappa p_\ve}(1+o_\ve(1))\int_{\mathbb R} v { \frac{e^{\eta_0}}{|x|^b}}\,dx+\frac{a^2}{\bar\kappa p_\ve} (1+o_\ve(1))\int_{\mathbb R} v\LL_{\frac12,0}v \,dx \right)^\frac{p_\ve}{2}.
   \end{align*}
  Taking $\ve\downarrow 0$, and by \eqref{conv-1}, we arrive to
   $$\frac{ 1}{\bar \kappa }\int_{\mathbb R} e^{av} {\frac{e^{\eta_0}}{|x|^b}}\,dx\leq
   \exp\left\{{\frac{a}{\bar \kappa} \int_{\mathbb R}  v {\frac{e^{\eta_0}}{|x|^b}}\,dx+\frac{a^2}{2\bar\kappa}  \int_{\mathbb R} v\LL_{\frac12,0}v \,dx}\right\}.$$
   Replacing $av$ by $v$ we get that
    \begin{align}\label{main-inequality}
    \log\left(\frac{1}{\bar \kappa} \int_{\mathbb R} e^{v} {\frac{e^{\eta_0}}{|x|^b}}\,dx\right)
    \leq  \frac{1}{\bar \kappa} \int_{\mathbb R}  v {\frac{e^{\eta_0}}{|x|^b}}\,dx+\frac{1}{2\bar\kappa}  \int_{\mathbb R} v\LL_{\frac12,0}v \,dx  , \end{align}
     for every $v\in C_c^\infty(\R)$. By density arguments we see that the above inequality holds in the space
   \begin{equation}\label{Xb}
   \mathcal X_b:=\left\{ v\in L^1_{loc}(\R):  \int_{\mathbb R} | v|  {\frac{e^{\eta_0}}{|x|^b}}\,dx+\frac{1}{2\bar\kappa}  \int_{\mathbb R} v\LL_{\frac12,0}v \,dx <\infty\right\}.
   \end{equation}

           Note that if constant functions were  in  $\mathcal X_b$ we would have
          by  taking $v\equiv C_0>>1$ in \eqref{main-inequality} that $$C_0+\log  {\frac{\kappa_0}{\bar\kappa}}\leq C_0 {\frac{\kappa_0}{\bar \kappa}},$$ and hence {$\bar \kappa=\kappa_0$} (as we already know that  {$\bar\kappa\geq\kappa_0$}). Finally, recalling from Theorem \ref{thm-limit-function} that $\kappa=2\pi(1-b)$, Theorem \ref{main-theorem} is proved  taking into account the shift of notation \eqref{shift}, up to the statement that constants are in  $\mathcal X_b$. This will be shown  in Lemma \ref{const}.
            \qed

%We denote
%\begin{equation*}
%%d\mu_b:=\frac{e^\eta}{|x|^b}\,dx.
%\end{equation*}
%When $b=0$, $\eta$ is a solution to
%$$(-\D)^\frac12\eta= e^\eta\quad\text{in }\R,\quad 2 \pi=\kappa:= \int_\R  e^\eta \,dx<\infty.  $$

%\begin{remark}
%The equality $\bar\kappa=\kappa$ implies that
%$$ \int_\R \frac{(u_\ve)^{p_\ve}}{|x|^{\beta_\ve p_\ve}}\,dx\to \int_\R\frac{e^\eta}{|x|^b}\,dx,$$
%and therefore, equality holds in \eqref{conv-1}.
%\end{remark}
%It follows from Proposition \ref{thm2} \textcolor{green}{which one?} {\color{blue}  This was written for $n\geq 2$ in older version of the file.... if this is also true for $n=1$, then we can get a similar identity  for the limit function $\eta$ as mentioned in Theorem \ref{thm-limit-function} } that there exists $\tau_\ve>0$ such that $$u_\ve\Big(\frac{\tau_\ve}{x}\Big)=|x|^{1-2\gamma_\ve-2\alpha_\ve}u_\ve(\tau_\ve x).$$
%    This implies that $$\int_{B_{\tau _\ve}} \frac{u_\ve^{p_\ve}}{|x|^{\beta_\ve p_\ve}}\,dx    = \int_{B^c_{\tau _\ve}} \frac{u_\ve^{p_\ve}}{|x|^{\beta_\ve p_\ve}}\,dx.$$
%    Then clearly $\tau_\ve\not\to0$. Moreover, if $\tau_\ve\to\infty$ then necessarily $\bar\kappa\geq 2\kappa$, which is impossible. In conclusion, $\tau_\ve\to \tau_0\in  (0,\infty)$.

\begin{lemma}\label{const} Constant functions are in the space $\mathcal X_b$ described  by  \eqref{Xb}.
\end{lemma}

\begin{proof}
  Let us first construct a sequence of functions $\psi_k\in C_c^\infty(\R)$  such that $\psi_k\to1$ in $C^0_{loc}(\R)$ and $\|(-\D)^\frac14 \psi_k\|_{L^2(\R)}\to0$.
   To this end we fix $\vp\in C_c^\infty(B_2)$ such that $0\leq\vp\leq1$, $\vp\equiv1$ on $B_1$. For $k\geq1$ we set $$\psi_k(x):=\frac1k\sum_{j=1}^k\vp_j(x),\quad \vp_j(x):=\vp\Big(\frac{x}{5 ^j}\Big).$$
   It follows easily that $\psi_k\to1$ in $C^0_{loc}(\R)$. For $|x|\leq\frac{ 5^j}{2}$ we see that $$|(-\D)^\frac12\vp_j(x)|\leq C 5^{-j}.$$  Using integration by parts one obtains
   $$\int_\R|(-\D)^\frac14\psi_k|^2\,dx=\int_\R \psi_k(-\D)^\frac12\psi_k \,dx,\quad \int_\R\vp_i(-\D)^\frac12\vp_j \,dx=\int_\R\vp_j(-\D)^\frac12\vp_i \,dx.$$
    Therefore, as $\|(-\D)^\frac14 \vp _i\|_{L^2(\R)}=\|(-\D)^\frac14 \vp\|_{L^2(\R)}$,
    \begin{align*}  \int_\R \psi_k(-\D)^\frac12\psi_k \,dx&=\frac{1}{k^2}\sum_{i,j=1}^k \int_\R\vp_i(-\D)^\frac12\vp_j \,dx \\&=\frac{1}{k^2}\sum_{i=1}^k \int_\R\vp_i(-\D)^\frac12\vp_i \,dx+\frac{2}{k^2}\sum_{i=1}^{k-1}\sum_{j=i+1}^k \int_\R\vp_i(-\D)^\frac12\vp_j \,dx
    \\&\leq \frac1k \|(-\D)^\frac14 \vp\|^2_{L^2(\R)}+\frac{2}{k^2}\sum_{i=1}^{k-1}\sum_{j=i+1}^k \int_{\{|x|\leq 2(5)^i\}}|(-\D)^\frac12\vp_j(x)| \,dx
    \\ &\leq \frac C k +\frac{C}{k^2}\sum_{i=1}^{k-1}\sum_{j=i+1}^k 5^{i-j}\\&\leq\frac Ck.
    \end{align*}
    Thus,   $\|(-\D)^\frac14 \psi_k\|_{L^2(\R)}\to0$. Also note that by dominated convergence theorem,  $$\int_\R|1-\psi_k|{\frac{e^{\eta_0}}{|x|^b}}\,dx\xrightarrow{k\to\infty}0,$$
and this finishes  the proof of the lemma.
\end{proof}

We conclude the paper by proving Remark \ref{remark-nottrue} from the introduction.

\begin{proof}[Proof of Remark \ref{remark-nottrue}] We fix $b\in(-1,0)$ and  two nonnegative, even  functions $\vp_1\in C_c^\infty(-1,1)$ and $\vp_2\in C_c^\infty(-2,2)$ such that $0\leq \vp_1,\vp_2\leq 1$ and  $$\vp_1\equiv 1\quad\text{on }\left(-\frac34,\frac34\right),\quad \vp_2\equiv 1\quad\text{on }(-1,1). $$
 For $ t>0$ small we set
 \begin{align*}
 \psi_ t(s)=\begin{cases}     1-\vp_2(\frac s t)&\quad\text{if }0\leq s\leq\frac12,\\ \vp_1(s)&\quad\text{if }s\geq\frac12,
 \end{cases}
 \end{align*} and
 $$v_ t(x)=\frac{1}{\sqrt\pi}\left(\log\frac1 t\right)^{-\frac12}\left(  \vp_2\Big(\frac{|x|}{ t}\Big)\log\frac1 t +\psi_ t(|x|)\log\frac{1}{|x|}\right).$$
 It has been shown in \cite[Proposition 2.1]{Hyd-MT} that $$\|(-\D)^\frac14 v_ t\|^2_{L^2(\R)}\leq 1+C\left(\log\frac1 t\right)^{-1}.$$
 Then setting
 $$\tilde v_ t(x)=2(1-b)\sqrt{\pi}\left(\log\frac1 t\right)^\frac12 v_ t(x-1),$$
  and observing that $1-b$ is uniformly bounded (since $b\in(-1,0)$ is fixed),  we see that
  $$\|(-\D)^\frac14 \tilde v_ t\|^2_{L^2(\R)}\leq 4\pi(1-b)^2\log\frac1 t+C.$$
  Moreover, by recalling that  $(-\D)^\frac14  ((-\D)^\frac14 \tilde v_t)= (-\D)^\frac12 \tilde v_t$ and duality we have that
  $$\|(-\D)^\frac14 \tilde v_ t\|^2_{L^2(\R)}=   \int_{\mathbb R} v(-\Delta_{\mathbb R})^{\frac{1}{2}}v \,dx.$$

  Notice that $$\tilde v_ t(x)=2(1-b)\log\frac1 t\quad\text{on }B_ t(1),$$ and therefore,
  $$\int_\R e^{\tilde v_ t}\frac{e^\eta}{|x|^b}\,dx\geq \frac12 e^{\eta(1)}\int_{1- t}^{1+ t} e^{\tilde v_t}\,dx=  e^{\eta(1)} \frac{t }{ t^{2(1-b)}} \geq \frac1C \frac{1}{ t^{1-2b}}.$$  On the other hand, it follows easily that
   $$\int_\R  \tilde v_ t \frac{e^\eta}{|x|^b}\,dx\leq C.$$

      Plugging these estimates in \eqref{conclusion}, and together with $\kappa=2\pi(1-b)$,  we deduce that $$(1-2b)\log\frac1 t\leq C+(1-b)\log\frac1 t,$$  which is impossible for $ t>0$ sufficiently small if $b<0$. \end{proof}

\appendix

\section{Extending the inequality}\label{subsection:extending}
We can extend the range of \eqref{parameter} to
\begin{equation}\label{parameter-modified}
\alpha\leq \beta\leq \alpha+\gamma, \quad \frac{n-2\gamma}{2}<\alpha<n.
\end{equation}
To see this, consider the inversion $\bar x=\frac{x}{|x|^2}$, $\bar y=\frac{y}{|y|^2}$ and make the change of variable, taking into account that $dx=|x|^{2n}d\bar x$, $dy=|y|^{2n}d\bar y$. Another important ingredient is formula \eqref{Kelvin}, that says
 \begin{equation*}
|\bar x-\bar y|=\frac{|x-y|}{|x|\, |y| }.
\end{equation*}
If we  set $u(x)=\bar u\left(\frac{x}{|x|^2}\right)$ then, on the one hand,
\begin{equation*}
\begin{split}\int_{\r^n}\int_{\r^n}\frac{(u(x)-u(y))^2}{|x-y|^{n+2\gamma}|x|^{{\alpha}}|y|^{{\alpha}}}\,dy\,dx&
=\int_{\r^n}\int_{\r^n}\frac{\left(\bar u(\frac{x}{|x|^2})-\bar u(\frac{y}{|y|^2})\right)^2}{|x-y|^{n+2\gamma}|x|^{{\alpha}}|y|^{{\alpha}}}\,dy\,dx\\
&=\int_{\r^n}\int_{\r^n}\frac{\left(\bar u(\bar x)-\bar u(\bar y)\right)^2}{|\bar x-\bar y|^{n+2\gamma}|\bar x|^{{\bar \alpha}}|\bar y|^{{\bar\alpha}}}\,d\bar y\,d\bar x,
\end{split}
\end{equation*}
where we have set
$$\bar \alpha=n-2\gamma-\alpha.$$
On the other hand, if we take
$$\beta p+\bar \beta p =2n,$$
which implies
$$p=\frac{2n}{n-2\gamma+2(\beta-\alpha)}=\frac{2n}{n-2\gamma+2(\bar \beta-\bar\alpha)},$$
 then we have that
\begin{equation*}
\int_{\r^n}\frac{|u(x)|^{p}}{|x|^{{\beta} {p}}}\,dx=\int_{\r^n}\frac{|\bar u(\bar x)|^{p}}{|\bar x|^{{\bar\beta} {p}}}\,d\bar x.
\end{equation*}
Thus, if the inequality is true for $\bar u$, $\bar \alpha$, $\bar \beta$, then it is also valid for $u$, $\alpha$ and $\beta$ in the range \eqref{parameter-modified}.

\end{document}